\documentclass{amsart}
\usepackage{amsmath,amssymb,mathrsfs,hyperref,color}

\usepackage{subfigure}
\usepackage{float} 
\usepackage{times}
\usepackage{tikz}
\usepackage{paralist}

\makeatletter
\def\setliststart#1{\setcounter{\@listctr}{#1}%
  \addtocounter{\@listctr}{-1}}
\makeatother

\usepackage{calc}
 \newtheorem{The}{Theorem}[section]
 \newtheorem{Cor}[The]{Corollary}
 \newtheorem{Lem}[The]{Lemma}
 \newtheorem{Pro}[The]{Proposition}
 \theoremstyle{definition}
 \newtheorem{defn}[The]{Definition}
 \theoremstyle{remark}
 \newtheorem{Rem}[The]{Remark}
 \newtheorem{ex}[The]{Example}
 \numberwithin{equation}{section}
\newcommand{\T}{\mathbb{T}}
\newcommand{\R}{\mathbb{R}}
\newcommand{\Z}{\mathbb{Z}}
\newcommand{\N}{\mathbb{N}}

\newcommand{\Singu}{\mbox{\rm Sing$\,(u)$}}

\title[Generalized characteristics and Lax-Oleinik operators]{Generalized characteristics and Lax-Oleinik operators: global theory}
\author{Piermarco Cannarsa \and Wei Cheng}
\address{Dipartimento di Matematica, Universit\`a di Roma ``Tor Vergata'',
Via della Ricerca Scientifica 1, 00133 Roma, Italy}
\email{cannarsa@mat.uniroma2.it}
\address{Department of Mathematics, Nanjing University,
Nanjing 210093, China}
\email{chengwei@nju.edu.cn}
\date{\today}
\begin{document}
\maketitle


\begin{abstract}
	For autonomous Tonelli systems on $\R^n$, we develop an intrinsic proof of the existence of generalized characteristics using sup-convolutions. This approach, together with convexity estimates for the fundamental solution, leads to new results such as the global propagation of singularities along generalized characteristics.
\end{abstract}

\section{Introduction}
Let $L(x,v)$ be a Tonelli Lagrangian on $\R^n$  ($L:\R^n\times\R^n\to\R$ is a function of class $C^2$, strictly convex in the fibre, with superlinear growth with respect to $v$), and let $H(x,p)$ be the associated  Hamiltonian given by the Fenchel-Legendre transform. The study of the regularity properties of the viscosity solutions  of the  Hamilton-Jacobi equation
\begin{equation}\label{eq:HJE}
H(x,Du(x))=0\qquad(x\in\R^n)
\end{equation}
is extremely important for several reasons. In the last two or three decades, remarkable progress in the broad area of Hamiltonian dynamical systems was achieved by Mather's theory for Tonelli  systems on compact manifolds in Lagrangian formalism~\cite{Mather91,Mather93}, and Fathi's weak KAM theory in Hamiltonian formalism~\cite{Fathi-book,Fathi-Siconolfi2004}. Both these theories  succeeded in the analysis of some hard dynamical problems such as Arnold diffusion. However, a general variational setting that applies to all the above aspects of Hamiltonian dynamics still has to be developed. As is well known, in both Mather's and Fathi's theories various global minimal sets in a variational sense---such as  Mather's set, Aubry's set and Ma\~n\'e's set---play a crucial role. Similarly, if we want to study the behavior of an orbit after it loses minimality, then we have to face the hard problem of dealing with the cut loci and singular sets of the associated viscosity solutions.

Although significant contributions investigating the singularities of viscosity solutions were already given in \cite{Cannarsa-Soner} and \cite{Ambrosio-Cannarsa-Soner}, the current approach to this problem goes back to \cite{alca99}, where the propagation of singularities was studied for general semiconcave functions. Since any viscosity solution $u$ of \eqref{eq:HJE} is locally semiconcave (with linear modulus), we have that singularities propagate along Lipschitz arcs starting from  any singular point $x$ of $u$ at which the superdifferential $D^+u(x)$ satisfies the condition
$$
\partial D^+u(x)\setminus D^*u(x)\not=\varnothing,
$$
where $\partial D^+u(x)$ denotes the topological boundary of $D^+u(x)$ and $D^*u(x)$ the set of all reachable gradients of $u$ at $x$. 
A more specific approach to the problem was developed in \cite{Albano-Cannarsa} by solving the generalized characteristic inclusion
$$
\dot{\mathbf{x}}(s)\in\mathrm{co}\, H_p\big(\mathbf{x}(s),D^+u(\mathbf{x}(s))\big),\quad \text{a.e.}\;s\in[0,\tau]\,.
$$
More precisely, if the initial point $x_0$ belongs to the singular set of $u$, hereafter denoted by $\Singu$, and is not a critical point of $u$ relative to $H$, i.e.,
$$
0\not\in\mathrm{co} \,H_p(x_0,D^+u(x_0))\,,
$$
then it was proved in \cite{Albano-Cannarsa} that there exists a nonconstant singular arc $\mathbf{x}$ from $x_0$ which is a generalized characteristic. The study of the local propagation of singularities along generalized characteristics was later refined in \cite{Yu} and \cite{Cannarsa-Yu}. For weak KAM solutions, local propagation results were obtained in \cite{Chen-Cheng} and the Lasry-Lions regularization procedure was applied in  \cite{Cannarsa-Cheng2} to analyze the critical points of Mather's barrier functions.
An interesting interpretation of the above singular curves as part of the flow of fluid particles has been recently proposed in \cite{Khanin-Sobolevski}  (see also \cite{Stromberg-Ahmadzadeh} for related results).

Returning to our dynamical motivations, in this paper we try to give an intrinsic interpretation of generalized characteristics and study the relevant global properties of such curves. For this purpose, we use the Lax-Oleinik semigroups $T^{\pm}_t$ (see, e.g. \cite{Fathi-book}) defined as follows: 
\begin{gather*}
T^+_tu_0(x):=\sup_{y\in\R^n}\{u_0(y)-A_t(x,y)\},\\
T^-_tu_0(x):=\inf_{y\in\R^n}\{u_0(y)+A_t(y,x)\},
\end{gather*}
where $u_0:\R^n\to\R$ is a continuous function and $A_t(x,y)$ is the fundamental solution of \eqref{eq:HJE}. These operators can be also derived from the Moreau-Yosida approximations in convex analysis (\cite{Attouch}) or the Lasry-Lions regularization technique based on sup- and inf- convolutions (\cite{Lasry-Lions}, \cite{Villani}).

By analyzing the maximizers $y_t$,  for sufficiently small $t>0$,  in the sup-convolution giving $T^+_tu_0(x)$ we obtain the global propagation of singularities which represents the main result of this paper. For such a result  we need  the following assumptions.
\begin{enumerate}[(L1)]
\item {\em Uniform convexity}: There exists a nonincreasing function $\nu:[0,+\infty)\to(0,+\infty)$ such that 
\begin{equation*}
L_{vv}(x,v)\geqslant \nu(|v|)I
\end{equation*}
for all  $(x,v)\in\R^n\times \R^n$.

\item {\em Growth conditions}: There exist two superlinear functions $\theta,\overline{\theta}:[0,+\infty)\to[0,+\infty)$ and a constant $c_0>0$ such that 
$$
\overline{\theta}(|v|)\geqslant L(x,v)\geqslant\theta(|v|)-c_0\qquad\forall  (x,v)\in\R^n\times \R^n.
$$

\item {\em Uniform regularity}: There exists a nondecreasing function $K:[0,+\infty)\to[0,+\infty)$ such that, for every multindex $|\alpha|=1,2$,
\begin{equation*}
|D^\alpha L(x,v)|\leqslant K(|v|)\qquad\forall  (x,v)\in\R^n\times \R^n,
\end{equation*}
\end{enumerate}
Recalling that $\Singu$ stands for the singular set of $u$ we now proceed to state our

\begin{description}
  \item[Propagation result] Let $L$ be a Lagrangian on $\R^n$ satisfying conditions {\rm (L1)-(L3)}, and let $u$ be a globally Lipschitz semiconcave solution of the Hamilton-Jacobi equation \eqref{eq:HJE}, where $H$ is the Hamiltonian associated with $L$. If $x$ belongs to $\Singu$, then there exists a generalized characteristic $\mathbf{x}:[0,+\infty)\to\R^n$ such that $\mathbf{x}(0)=x$ and $\mathbf{x}(s)\in\Singu$ for all $s\in [0,+\infty)$.
\end{description}

We observe that the study of the global propagation of singularities is much more difficult than the local one. Indeed, to this date the only known results concern geodesic systems, see \cite{Albano2014_2} and \cite{acns}. More precisely, \cite{Albano2014_2} studies the global propagation of the so-called $C^1$-singular support of solutions which---in this case---coincides with the closure of Sing$(u)$, while \cite{acns} investigates the propagation of genuine singularities. For time dependent problems, global propagation was addressed in \cite{Albano2014_1} for the $C^1$-singular support of solutions and, recently, in \cite{Cannarsa-Mazzola-Sinestrari} for singularities of solutions to eikonal equations.

It is worth mentioning that, when considering geodesic systems on Riemannian manifolds, the method of generalized characteristics (or generalized gradient flow) has been successfully applied to reveal topological relations between  a compact domain $\Omega$ and the cut locus enclosed in $\Omega$ (\cite{acns}). Such relations depend on global results for the propagation of singularities along the associated generalized characteristics. An analogous relation between the Aubry set and the cut locus can be deduced from our results as  we will show in the near future. In preparation for such an application, in this paper we have included section~\ref{se:torus} where our global propagation result is adapted to the $n$-dimensional flat torus. 

For the proof of the above theorem we need regularity results for the value function of the action functional (also called fundamental solution of \eqref{eq:HJE} in \cite{McEneaney-Dower})
$$
A_{t}(x,y)=\inf_{\gamma\in\Gamma^t_{x,y}}\int^{t}_{0}L(\gamma(s),\dot{\gamma}(s))ds\quad (t>0\,,\;x,y\in\R^n)
$$
where
$$
\Gamma^t_{x,y}=\{{\gamma\in W^{1,1}([0,t],\R^n): \gamma(0)=x,\gamma(t)=y}\}.
$$
More precisely, we need 
\begin{description}
  \item[Convexity and local $C^{1,1}$ regularity results] Suppose $L$ is a Tonelli Lagrangian satisfying {\rm (L1)-(L3)}. Then the following properties hold true.
  \begin{enumerate}[(a)]
  \item For any $\lambda>0$,  there exists $t_\lambda>0$ such that, for any $x\in\R^n$, the function $(t,y)\mapsto A_t(x,y)$ is semiconvex on the  cone
\begin{equation*}
S_\lambda(x,t_\lambda):=\big\{(t,y)\in\R\times\R^n~:~0<t< t_\lambda,\; |y-x|<\lambda t\big\}\,,
\end{equation*}
that is,  there exists $ C_\lambda>0$ such that for all $x\in\R^n$, all $(t,y)\in S_\lambda(x,t_\lambda)$,  all $h\in[0,t/2)$, and  all $z\in  B(0,\lambda t)$ we have that
$$
A_{t+h}(x,y+z)+A_{t-h}(x,y-z)-2A_t(x,y)\geqslant - \frac{C_ \lambda}{t}(h^2+|z|^2).
$$
  \item For all $t\in(0,t_{\lambda}]$, $A_t(x,\cdot)$ is uniformly convex on $B(x,\lambda t)$, that is,   there exists  $C_{\lambda}'>0$ such that for all $x\in\R^n$, all $y\in B(x,\lambda t)$, and all $z\in  B(0,\lambda t)$ we have that
$$
A_{t}(x,y+z)+A_{t}(x,y-z)-2A_t(x,y)\geqslant \frac{C'_{\lambda}}{t}|z|^2.
$$
  \item For any $x\in\R^n$ the functions $(t,y)\mapsto A_t(x,y)$ and $(t,y)\mapsto A_t(y,x)$ are of class $C^{1,1}_{\text{loc}}$ on the cone $S_{\lambda}(x,t_\lambda)$ defined above. 
  \end{enumerate}
\end{description}
Similar regularity results were obtained in \cite{Bernard2012} by a different approach, under more restrictive structural assumptions than those we consider in this paper. 

This paper is organized as follows. In section 2, we review basic properties of viscosity solution of Hamilton-Jacobi equations. In section 3, we discuss connections between sup-convolutions and generalized characteristics and we give our global result on the propagation of singularities along generalized characteristics. Moreover, for Tonelli systems, we  adapt the above results to the $n$-dimensional torus. The paper contains four appendices that contain technical results and background material which is useful for our approach: in the first one we give a uniform bound for minimizers of the action functional following  \cite{Ambrosio-Ascenzi-Buttazzo} and \cite{Dal-Maso-Frankowska}, in the second one we give detailed proofs of all the required regularity results for the fundamental solution, in the third one we adapt the construction of generalized characteristics from \cite{Albano-Cannarsa} to the present context, in the fourth one we provide a global semiconcavity estimate for the weak KAM solution on $\R^n$ constructed in \cite{Fathi-Maderna}.
\medskip

\noindent{\bf Acknowledgments} This work was partially supported by the Natural Scientific Foundation of China (Grant No. 11271182 and No. 11471238), the  National Basic Research Program of China (Grant No. 2013CB834100), a program PAPD of Jiangsu Province, China, and the National Group for Mathematical Analysis, Probability and Applications (GNAMPA) of the Italian Istituto Nazionale di Alta Matematica ``Francesco Severi''. The authors thank Albert Fathi for the discussion during his visit in Nanjing. The second author is grateful to Cui Chen, Liang Jin for helpful discussions, and also Jun Yan for the communications and the hospitality during the visit to Fudan University in 2014-15.

\section{Preliminaries}

\subsection{Semiconcave functions}

Let $\Omega\subset\R^n$ be a convex set. We recall that a function $u:\Omega\rightarrow\R$ is said to be {\em semiconcave} (with linear modulus) if there exists a constant $C>0$ such that
\begin{equation}\label{eq:SCC}
\lambda u(x)+(1-\lambda)u(y)-u(\lambda x+(1-\lambda)y)\leqslant\frac C2\lambda(1-\lambda)|x-y|^2
\end{equation}
for any $x,y\in\Omega$ and $\lambda\in[0,1]$.  Any constant $C$ that satisfies the above inequality  is called a {\em semiconcavity constant} for $u$ in $\Omega$.

A function $u:\Omega\rightarrow\R$ is said to be {\em semiconvex} if $-u$ is semiconcave. 

When $u:\Omega\to\R$ is continuous, it can be proved that $u$ is semiconcave with constant $C$ if and only if   
$$
u(x)+u(y)-2u\left(\frac{x+y}2\right)\leqslant \frac C2|x-y|^2
$$
for any $x,y\in\Omega$.

Hereafter, we assume that $\Omega$ is a nonempty open subset of $\R^n$.

We recall that a function $u:\Omega\rightarrow\R$ is said to be {\em locally semiconcave} (resp. {\em locally semiconvex}) if for each $x\in\Omega$ there exists an open ball $B(x,r)\subset\Omega$ such that $u$ is a semiconcave (resp. semiconvex) function on $B(x,r)$.

Let $u:\Omega\subset\R^n\to\R$ be a continuous function. We recall that, for any $x\in\Omega$, the closed convex sets
\begin{align*}
D^-u(x)&=\left\{p\in\R^n:\liminf_{y\to x}\frac{u(y)-u(x)-\langle p,y-x\rangle}{|y-x|}\geqslant 0\right\},\\
D^+u(x)&=\left\{p\in\R^n:\limsup_{y\to x}\frac{u(y)-u(x)-\langle p,y-x\rangle}{|y-x|}\leqslant 0\right\}.
\end{align*}
are called the {\em (Dini) subdifferential} and {\em superdifferential} of $u$ at $x$, respectively.

Let now $u:\Omega\to\R$ be locally Lipschitz. We recall that a vector $p\in\R^n$ is said to be a {\em reachable} (or {\em limiting}) {\em gradient} of $u$ at $x$ if there exists a sequence $\{x_n\}\subset\Omega\setminus\{x\}$, converging to $x$, such that $u$ is differentiable at $x_k$ for each $k\in\N$ and
$$
\lim_{k\to\infty}Du(x_k)=p.
$$
The set of all reachable gradients of $u$ at $x$ is denoted by $D^{\ast}u(x)$.

Now we list some   well known properties of the superdifferential  of a semiconcave function on $\Omega\subset\R^n$ (see, e.g., \cite{Cannarsa-Sinestrari} for the proof).

\begin{Pro}\label{basic_facts_of_superdifferential}
Let $u:\Omega\subset\R^n\to\R$ be a semiconcave function and let $x\in\Omega$. Then the following properties hold.
\begin{enumerate}[\rm {(}a{)}]
  \item $D^+u(x)$ is a nonempty compact convex set in $\R^n$ and $D^{\ast}u(x)\subset\partial D^+u(x)$, where  $\partial D^+u(x)$ denotes the topological boundary of $D^+u(x)$.
  \item The set-valued function $x\rightsquigarrow D^+u(x)$ is upper semicontinuous.
  \item If $D^+u(x)$ is a singleton, then $u$ is differentiable at $x$. Moreover, if $D^+u(x)$ is a singleton for every point in $\Omega$, then $u\in C^1(\Omega)$.
  \item $D^+u(x)=\mathrm{co}\, D^{\ast}u(x)$.
\end{enumerate}
\end{Pro}

\begin{Pro}[\cite{Cannarsa-Sinestrari}]
\label{criterion-Du_semiconcave2}
Let $u:\Omega\to\R$ be a continuous function. If there exists a constant $C>0$ such that, for any $x\in\Omega$, there exists $p\in\R^n$ such that
\begin{equation}\label{criterion_for_lin_semiconcave}
u(y)\leqslant u(x)+\langle p,y-x\rangle+\frac C2|y-x|^2,\quad \forall y\in\Omega,
\end{equation}
then $u$ is semiconcave with constant $C$ and $p\in D^+u(x)$.
Conversely,
if $u$ is semiconcave  in $\Omega$ with constant $C$, then \eqref{criterion_for_lin_semiconcave} holds for any $x\in\Omega$ and $p\in D^+u(x)$.
\end{Pro}

A point $x\in\Omega$ is called a {\em singular point} of $u$ if $D^+u(x)$ is not a singleton. The set of all singular points of $u$, also called the {\em singular set} of $u$, is denoted by $\Singu$.

\subsection{Tonelli Lagrangians}
In this paper, we concentrate on Lagrangians on Euclidean configuration space $\R^n$. We say that a function $\theta:[0,+\infty)\to[0,+\infty)$ is {\em superlinear} if $\theta(r)/r\to+\infty$ as $r\to+\infty$. 

\begin{defn}\label{defn_generalized_Tonelli_function}
	A function $F:\R^n\times\R^n\to\R$ is called a {\em generalized Tonelli function} if $F$ is a function of class $C^2$ that satisfies the following conditions:
\begin{enumerate}[(T1)]
\item {\em Uniform convexity}: There exists a nonincreasing function $\nu:[0,+\infty)\to(0,+\infty)$ such that 
\begin{equation*}
F_{vv}(x,v)\geqslant \nu(|v|)I\qquad\forall  (x,v)\in\R^n\times \R^n.
\end{equation*}
\item {\em Growth condition}: There exist two superlinear function $\theta,\overline{\theta}:[0,+\infty)\to[0,+\infty)$ and a constant $c_0>0$ such that 
$$
\overline{\theta}(|v|)\geqslant F(x,v)\geqslant\theta(|v|)-c_0\qquad\forall  (x,v)\in\R^n\times \R^n.
$$
\item {\em Uniform regularity}: There exists a nondecreasing function $K:[0,+\infty)\to[0,+\infty)$ such that, for every multindex $|\alpha|=1,2$,
\begin{equation*}
|D^\alpha F(x,v)|\leqslant K(|v|)\qquad\forall  (x,v)\in\R^n\times \R^n,
\end{equation*}
\end{enumerate} 
\end{defn}

The convex conjugate of a superlinear function $\theta$ is defined as 
\begin{equation}\label{eq:convex_conj_theta}
	\theta^*(s)=\sup_{r\geqslant0}\{rs-\theta(r)\} s\qquad \forall s\geqslant 0.
\end{equation}
In view of the superlinear growth  of $\theta$ it is clear that $\theta^*$ is well defined and satisfies
\begin{equation}
\label{eq:young}
\theta(r)+\theta^*(s)\geqslant rs\qquad \forall r,s\geqslant0,
\end{equation}
which in turn can be used to show that $\theta^*(s)/s\to\infty$ as $s\to\infty$.

\begin{defn}
	A function $L:\R^n\times\R^n\to\R$ is called a {\em Tonelli Lagrangian} if $L$ is a generalized Tonelli function as in Definition \ref{defn_generalized_Tonelli_function}. If $L$ is a Tonelli Lagrangian, the associated Hamiltonian $H$ is the Fenchel-Legendre dual of $L$ is defined by
\begin{equation}\label{eq:H}
H(x,p)=\sup_{v\in\R^n}\big\{\langle p,v \rangle-L(x,v)\big\}\qquad(x,p)\in \R^n\times\R^n\,.
\end{equation}
The Hamiltonian $H$ is called a {\em Tonelli Hamiltonian} if $H$ is a generalized Tonelli function.
We denote by (L1)-(L3) (resp. (H1)-(H3)) the corresponding conditions of a Tonelli Lagrangian $L$ (resp. Tonelli Hamiltonian $H$).
\end{defn}

\begin{ex}
A typical example of a Tonelli Lagrangian $L$ is the one of mechanical systems which has the form
$$
L(x,v)=f(x)(1+|v|^2)^{\frac q2}+V(x),\quad (x,v\in\R^n),
$$
where $q>1$,  $f$ and $V$ are smooth functions on $\R^n$ with bounded derivatives up to the second order, and $\inf_{\R^n}f>0$.
\end{ex}

\begin{Lem}\label{uniform_condition_H}
If $L$ is a Tonelli Lagrangian with $H$ its Fenchel-Legendre dual, then $H$ is a Tonelli Hamiltonian.
\end{Lem}

\begin{proof}
First, let $v_{x,p}\in\R^n$ be such that  $p=L_v(x,v_{x,p})$ or, equivalently, $v_{x,p}=H_p(x,p)$. By assumptions (L1)-(L3), we have that
\begin{align*}
	|p||v_{x,p}|\geqslant& \langle L_v(x,v_{x,p}),v_{x,p}\rangle\geqslant L(x,v_{x,p})-L(x,0)\geqslant\theta(|v_{x,p}|)-c_0-\bar{\theta}(0)\\
	\geqslant&(|p|+1)|v_{x,p}|-c_0-\bar{\theta}(0)-\theta^*(|p|+1).
\end{align*}
This shows that 
$$
|H_p(x,p)|=|v_{x,p}|\leqslant c_0+\bar{\theta}(0)+\theta^*(|p|+1)=C_1(|p|),
$$
where $C_1(r)=c_0+\bar{\theta}(0)+\theta^*(r+1)$. The estimates for the other derivatives in (H3) and (H1) can be proved by  using the relations
	\begin{align*}
		&H_x(x,p)=-L_x(x,H_p(x,p))\\
		&H_{pp}(x,p)=L^{-1}_{vv}(x,H_p(x,p))\\
		&H_{xp}(x,p)=-L_{xv}(x,H_p(x,p))H_{pp}(x,p)\\
		&H_{xx}(x,p)=-L_{xx}(x,H_p(x,p))-L_{xv}(x,H_p(x,p))H_{px}(x,p)
	\end{align*}
	and our conditions (L1)-(L3). This completes the verification of (H1) and (H3).
	
	To check (H2), we have that, for all $R\geqslant0$,
	$$
	H(x,p)\geqslant R|p|-\bar{\theta}(R),\quad \forall (x,p)\in\R^n\times\R^n
	$$
	by \eqref{eq:H} and (L2). Set $\theta_1(r)=\sup_{R>0}\{Rr-\bar{\theta}(R)\}$, $r\in(0,+\infty)$, which is well defined by (L2). Thus, $\theta_1$ gives the required superlinear function for (H2).
\end{proof}

\subsection{Hamilton-Jacobi equations.}
Suppose $H$ is the Hamiltonian associated with a Tonelli Lagrangian $L$ and consider the Hamilton-Jacobi equation
\begin{equation}\label{HJ1}
H(x,Du(x))=0\qquad (x\in\R^n).
\end{equation}
We recall that a continuous function $u$ is called a {\em viscosity subsolution} of equation
\eqref{HJ1} if, for any $x\in\R^n$,
\begin{align}\label{viscosity subsolution}
H(x,p)\leqslant0,\quad\forall p\in D^+u(x)\,.
\end{align}
Similarly, $u$ is a {\em viscosity supersolution} of equation \eqref{HJ1} if, for any $x\in\R^n$,
\begin{align}\label{viscosity supersolution}
H(x,p)\geqslant0,\quad\forall p\in D^-u(x)\,.
\end{align}
Finally, $u$ is called a {\em viscosity solution} of equation \eqref{HJ1}, if it is both a viscosity subsolution and a supersolution.

Throughout this paper we will be concerned with solutions of the above equation that are Lipschitz continuous and semiconcave on $\R^n$. The existence of such solution is the object of the following proposition which is essentially a consequence of the existence theorem of \cite{Fathi-Maderna} and the semiconcavity results of this paper (see also \cite{Cannarsa-Sinestrari}  and \cite{Rifford}).

\begin{Pro}\label{Ext_and_reachable}
Let $L$ be a Tonelli Lagrangian and let $H$ be the associated Hamiltonian.
Then there exists a constant $c(H)\in\R$ such that the Hamilton-Jacobi equation
\begin{equation}\label{eq:crtical_H_J_eqn}
	H(x,Du(x))=c,\quad x\in\R^n,
\end{equation}
admits a viscosity solution $u:\R^n\to\R$  for $c=c(H)$ and does not admit any such solution for $c<c(H)$. Moreover, $u$ is globally Lipschitz continuous and semiconcave on $\R^n$.
\end{Pro}
The proof of Proposition~\ref{Ext_and_reachable} is given in Appendix~\ref{appendix:B}.

Let us consider the class of {\em admissible arcs}
$$
\mathcal{A}_{t,x}=\{\xi\in W^{1,1}([0,t];\R^n): \xi(t)=x\}
$$
where $W^{1,1}([a,b];\R^n)$ denotes the space of all absolutely continuous $\R^n$-valued functions on $[a,b]$, where $-\infty<a<b<+\infty$. The functional 
\begin{equation}\label{eq:action-functional}
J_t(\xi):=\int^t_0L(\xi(s),\dot{\xi}(s))\ ds+u_0(\xi(0)),\quad \xi\in\mathcal{A}_{t,x},
\end{equation}
where $u_0\in C(\R^n)$ is the {\em initial cost}, is usually called the {\em action functional}. A classical problem in the calculus of variations is
\begin{equation}\label{CV}
\text{to minimize}\ J_t\ \text{over all arcs}\ \xi\in\mathcal{A}_{t,x}.\tag{CV$_{t,x}$}
\end{equation}
We define the associated value function
\begin{equation}\label{eq:value-function}
u(t,x)=\min_{\xi\in\mathcal{A}_{t,x}}J_t(\xi).
\end{equation}
It is known that $u(t,x)$ is a viscosity solution of the Cauchy problem
\begin{equation}\label{eq:HJ_evolution}
\begin{cases}
u_t(t,x)+H(x,\nabla_xu(t,x))=0,&t> 0, x\in\R^n\\
u(0,x)=u_0(x),& x\in\R^n,
\end{cases}
\end{equation}
where $\nabla_xu$ denotes the spatial gradient of $u$. From the uniqueness of viscosity solutions of \eqref{eq:HJ_evolution} if follows that, if the initial datum $u_0$ is a viscosity solution of \eqref{HJ1}, then the solution $u$ of \eqref{eq:HJ_evolution} is constant in time and coincides with $u_0$. In this case, because of the
translation invariance of problem \eqref{CV}, we have that, for all $t\geqslant 0$,
\begin{equation}\label{cv_t}
u_0(x)=\min_{\xi\in W^{1,1}([-t,0];\R^n)}\Big\{\int_{-t}^0L(\xi(s),\dot{\xi}(s))\ ds+u_0(\xi(-t))~:~\xi(0)=x\Big\}.
\end{equation}
Moreover, suppose $L$ satisfies conditions (L1)-(L3) and let $H$ be the associated Hamiltonian. Then we have the following result (see \cite{Cannarsa-Sinestrari} or \cite{Rifford}).
\begin{Pro}\label{reachable_grad_and_backward}
Let $u:\R^n\to\R$ be a viscosity solution of  \eqref{HJ1} and let  $x\in \R^n$. Then $p\in D^{\ast}u(x)$ if and only if there exists a unique $C^2$ curve $\gamma:(-\infty,0]\to \R^n$ with $\gamma(0)=x$ which is a minimizer of the problem in \eqref{cv_t} for every $t\geqslant 0$ and $p=L_v(x,\dot{\gamma}(0))$.
\end{Pro}

\subsection{Generalized characteristics}
The study of the structure of the singular set of a viscosity solution is a very important and hard one in many fields such as Riemannian geometry, optimal control, classical mechanics, etc. The dynamics of singularities can be described by using generalized characteristics.

\begin{defn}
A Lipschitz arc $\mathbf{x}:[0,T]\to\R^n,\,(T>0),$ is said to be a {\em generalized characteristic} of the Hamilton-Jacobi equation \eqref{HJ1} if $\mathbf{x}$ satisfies the differential inclusion
\begin{equation}\label{generalized_characteristics}
\dot{\mathbf{x}}(s)\in\mathrm{co}\, H_p\big(\mathbf{x}(s),D^+u(\mathbf{x}(s))\big),\quad \text{a.e.}\;s\in[0,T]\,.
\end{equation}
\end{defn}

A basic criterion for the propagation of singularities along generalized characteristics was given in \cite{Albano-Cannarsa} (see \cite{Cannarsa-Yu,Yu} for an improved version and simplified proof of this result).
\begin{Pro}[\cite{Albano-Cannarsa}]\label{criterion_on_gen_char}
Let $u$ be a viscosity solution of \eqref{HJ1} and let $x_0\in\R^n$. Then there exists a generalized characteristic  $\mathbf{x}:[0,T]\to\R^n$ with initial point $\mathbf{x}(0)=x_0$. Moreover, if
$x_0\in\Singu$, then $\tau\in(0,T)$ exists such that $\mathbf{x}(s)\in\Singu$ for all $s\in [0,\tau]$. Furthermore, if
\begin{equation}\label{condition_for_propagation_singularities}
0\not\in\mathrm{co} \,H_p(x_0,D^+u(x_0))\,,
\end{equation}
 then  $\mathbf{x}(s)\neq x_0$ for every $s\in[0,\tau]$.
\end{Pro}

Condition \eqref{condition_for_propagation_singularities} is the key point to guarantee the propagation of singularities along generalized characteristics. For the cell problem with $L$ in the form $$L(x,v)=\frac 12\langle A(x)v,v\rangle-V(x)+E,$$ a local propagation result can be obtained replacing assumption \eqref{condition_for_propagation_singularities} by the energy condition $E>\max_{x\in\R^n}V(x)$ (see \cite{Cannarsa-Cheng}).



\section{Generalized characteristics and Lax-Oleinik operators}
For any $t>0$, given $x, y\in\R^n$, we set
$$
\Gamma^t_{x,y}=\{{\xi\in W^{1,1}([0,t];\R^n): \xi(0)=x,\xi(t)=y}\}
$$
and define
\begin{equation}\label{fundamental_solution}
A_{t}(x,y)=\min_{\xi\in\Gamma^t_{x,y}}\int^{t}_{0}L(\xi(s),\dot{\xi}(s))ds\qquad (x, y\in\R^n).
\end{equation}
The existence of the above minimum is a well-known result in Tonelli's theory (see, for instance, \cite{Fathi-book}). Any $\xi\in\Gamma^t_{x,y}$ at which the minimum in \eqref{fundamental_solution} will be called a {\em minimizer} for $A_t(x,y)$ and such a minimizer $\xi$ is of class $C^2$ by classical results. In the PDE literature, $A_t(x,y)$ is also called the {\em fundamental solution} of \eqref{HJ1}, see, for instance, \cite{McEneaney-Dower}.

Let $L$ be a Tonelli Lagrangian satisfying {\rm (L1)-(L3)}  and let $H$ be the associated Hamiltonian. In this section we study the singularities of a Lipschitz continuous semiconcave solution $u$  of the Hamilton-Jacobi equation 
\begin{equation}\label{eq:HJ-general}
H(x,Du(x))=0,\quad x\in\R^n.
\end{equation}
The existence of such a solution is guaranteed by Proposition~\ref{Ext_and_reachable}.
	
\subsection{Lax-Oleinik operators}
For any Lipschitz continuous function  $u:\R^n\to\R$ we set
\begin{equation}\label{eq:Lip}
{\rm Lip} (u)=\sup_{y\neq x}\frac{|u(y)-u(x)|}{|y-x|}\,.
\end{equation}\label{eq:barrier}
For all  $(t,x)\in\R_+\times\R^n$  let
\begin{equation}
\label{eq:phipsi}
\phi^x_t(y)=u(y)-A_t(x,y)\quad\mbox{and}\quad \psi^x_t(y) =u(y)+A_t(y,x)\qquad(y\in\R^n)
\end{equation}
where $A_t$ is the fundamental solution of \eqref{eq:HJ-general}. The Lax-Oleinik operators $T^-_t$ and $T^+_t$ are defined as follows
\begin{gather}
T^+_tu(x)=\sup_{y\in\R^n}\phi^x_t(y)
,\quad x\in\R^n,
\label{L-L regularity_sup}\\
T^-_tu(x)=\inf_{y\in\R^n}\psi^x_t(y)
,\quad x\in\R^n.
\label{L-L regularity_inf}
\end{gather}
The functions $\phi^x_t$ and $\psi^x_t$ are also called {\em local barrier functions}. 

\begin{Lem}\label{sup_max}
Suppose $L$ is a Tonelli Lagrangian and
let $u$ be a Lipschitz function on $\R^n$. Then the supremum in \eqref{L-L regularity_sup} is attained for every $(t,x)\in\R_+\times\R^n$. Moreover, there exists a  constant $\lambda_0>0$, depending only on ${\rm Lip} (u)$, such that, for any $(t,x)\in\R_+\times\R^n$ and any maximum point $y_{t,x}$ of $\phi^x_t$, we have
\begin{equation}
\label{eq:maxbound}
	|y_{t,x}-x|\leqslant \lambda_0t\,.
\end{equation}
\end{Lem}

\begin{proof}
	Let $k_u={\rm Lip}(u)+1$. Then, for any $(t,x)\in\R_+\times\R^n$ and $y\in\R^n$,  \eqref{eq:young} yields
	\begin{align*}
		A_t(x,y)\geqslant&\inf_{\xi\in\Gamma^t_{x,y}}\int^t_0\theta(|\dot{\xi}(s)|)\ ds-c_0t\geqslant\inf_{\xi\in\Gamma^t_{x,y}}k_u\int^t_0|\dot{\xi}(s)|\ ds-\big(\theta^*(k_u)+c_0\big)t\\
		\geqslant& k_u|y-x|-\big(\theta^*(k_u)+c_0\big)t,
	\end{align*}
where $c_0$ is the constant in assumption (L2). Therefore
	\begin{align*}
		&\phi^x_t(y)-\phi^x_t(x)=u(y)-u(x)-A_t(x,y)+A_t(x,x)\\
		\leqslant&\; {\rm Lip}(u)|y-x|-k_u|y-x|+t\big(\theta^*(k_u)+c_0\big)+tL(x,0)\\
		\leqslant&-|y-x|+t\big(\theta^*(k_u)+c_0+K(0)\big)
	\end{align*} 
	where $K$ is given by condition (L3). Now, taking $\lambda_0=\theta^*(k_u)+c_0+K(0)$ it follows that
\begin{equation}
\label{eq:levelset}
	\Lambda^x_t:=\{y:\phi^x_t(y)\geqslant \phi^x_t(x)\}\subset\overline{B}(x,\lambda_0t).
\end{equation}
	Therefore $\Lambda^x_t$ is compact and the supremum in \eqref{L-L regularity_sup} is indeed a maximum. 
	Moreover, 
	\eqref{eq:maxbound} is a consequence of \eqref{eq:levelset}.
\end{proof}

A similar result holds for the inf-convolution defined in \eqref{L-L regularity_inf}. A more detailed study of the properties of inf/sup-convolutions can be found in \cite{Attouch} with respect to the quadratic Hamiltonian $H(p)=|p|^2/2$ and, consequently, the kernel $A_t(x,y)=\frac 1{2t}|x-y|^2$. This type of regularization, also called Moreau-Yosida regularization in convex analysis,  was developed into a well-known procedure by Lasry and Lions~\cite{Lasry-Lions}. In our context, we recover more information from the dynamical systems point of view by replacing quadratic kernels with the fundamental solutions. 

\subsection{Propagation of singularities}\label{Procedure_of_sup_convolution}
In this section, we will discuss the connection between sup-convolutions, singularities, and generalized characteristics. We begin our analysis with the local propagation of singularities of viscosity solutions along generalized characteristics. For  Tonelli systems under rather general conditions, a local propagation result was obtained in \cite{alca99} by a different method, without relating singular arcs to generalized characteristics. In the following lemma, we construct a singular arc starting from any singular point of the solution. A crucial point of this result is the fact that the interval $[0,t_0]$ on which the singular arc is defined turns out to be independent of the starting point $x$.  

\begin{Lem}\label{maximizer_singular}
Let $L$ be a Tonelli Lagrangian and let $H$ be the associated Hamiltonian. Suppose $u:\R^n\to\R$ is a Lipschitz continuous semiconcave viscosity solution of \eqref{eq:HJ-general}. Then there exists $t_0$ in $(0,1]$ such that, for all  $(t,x)\in(0,t_0]\times\R^n$, there is a unique maximum point $y_{t,x}$ of $\phi^x_t$ and  the curve
\begin{equation}\label{maximizer_arc}
\mathbf{y}(t):=\begin{cases}
x&\mbox{if}\quad t=0\\
y_{t,x}&\mbox{if}\quad t\in(0,t_0]
\end{cases}
\end{equation}
satisfies $\lim_{t\to0}\mathbf{y}(t)=x$.

 Moreover, if $x\in\Singu$, then $\mathbf{y}(t)\in\Singu$ for all $t\in(0,t_0]$. \end{Lem}
\begin{proof}
Let  $C_1>0$ be a semiconcavity constant for $u$ on $\R^n$ and let $\lambda_0$ be the positive constant in Lemma~\ref{sup_max}. By  Proposition~\ref{convexity_A_t} with $\lambda=1+\lambda_0$, we deduce that there exists $t_0\in(0,1]$ and a constant $C_2>0$ such that for every $(t,x)\in (0,t_0]\times\R^n$, every 
$y\in B(x,\lambda t)$, and every $z\in B(0,\lambda t)$ we have
 that
\begin{equation*}
A_{t}(x,y+z)+A_{t}(x,y-z)-2A_t(x,y)\geqslant \frac{C_2}{t}|z|^2.
\end{equation*}
Thus,  $\phi^x_t(y)=u(y)-A_t(x,y)$ is strictly concave on $\overline{B}(x,\lambda t)$ for all $t\in(0,t_0]$ provided that we further restrict $t_0$ in order to have 
\begin{equation}
\label{eq:t_0}
t_0<\frac{C_2}{C_1}.
\end{equation}
Then, for all such numbers $t$, there exists a unique maximum point $y_{t,x}$ of $\phi^x_t$ in $\overline{B}(x,\lambda t)$. In fact, $y_{t,x}$ is an interior point of $B(x,\lambda t)$ since, by Lemma \ref{sup_max}, we have that $|y_t-x|\leqslant\lambda_0 t$.

We now prove that $y_{t,x}$ is a singular point of $u$ for every $t\in(0,t_0]$.
Let $\xi_{t,x}\in\Gamma^t_{x,y_{t,x}}$ be the unique minimizer for $A_t(x,y_{t,x})$ and let
$$
p_{t,x}(s):=L_v(\xi_{t,x}(s),\dot{\xi}_{t,x}(s)),\quad s\in[0,t_0],
$$
be the associated dual arc. We claim that 
\begin{equation}\label{eq:criterion_singularity}
p_{t,x}(t)\in D^+u(y_{t,x})\setminus D^*u(y_{t,x}),
\end{equation}
which in turn yields $y_{t,x}\in\Singu$. Indeed, if $p_{t,x}(t)\in D^*u(y_{t,x})$, then by Proposition~\ref{reachable_grad_and_backward} there would exist a $C^2$ curve $\gamma_{t,x}:(-\infty,t]\to\R^n$ solving the minimum problem 
$$
\min_{\gamma\in W^{1,1}([\tau,t];\R^n)}\Big\{\int_{\tau}^tL(\gamma(s),\dot{\gamma}(s))\ ds+u(\gamma(\tau))~:~\gamma(t)=y_{t,x}\Big\}
$$
for all $\tau\leqslant t$. It is easily checked that $\gamma_{t,x}$ and $\xi_{t,x}$ coincide on $[0,t]$ since both of them are  extremal curves for $L$ and satisfy the same endpoint condition at $y_{t,x}$, i.e., 
 $$L_v(\xi_{t,x}(t),\dot{\xi}_{t,x}(t))=p_{t,x}(t)=L_v(\gamma_{t,x}(t),\dot{\gamma}_{t,x}(t)).$$ This leads to a contradiction since $x\in \Singu$ while $u$ should be smooth at $\gamma_{t,x}(0)$. Thus, \eqref{eq:criterion_singularity} holds true and $y_t\in\Singu$.
\end{proof}

Next, we proceed to show that the singular arc in Lemma~\ref{maximizer_singular} is a generalized characteristic. 

\begin{Lem}\label{uni_Lip_1}
Let $t_0$ and $\mathbf{y}$ be given by Lemma~\ref{maximizer_singular} for a given $x\in\R^n$. For any $t\in(0,t_0]$ let $\xi_{t,x}\in\Gamma^t_{x,\mathbf{y}(t)}$ be a minimizer for $A_t(x,\mathbf{y}(t))$. Then
\begin{equation}\label{eq:equi-Lipschitz}
\{\dot{\xi}_{t,x}(\cdot)\}_{t\in(0,t_0]}\ \text{is an equi-Lipschitz family}.
\end{equation}
\end{Lem}

\begin{proof}
Since $\mathbf{y}(t)\in\overline{B}(x,\lambda t)$ by Lemma \ref{sup_max}, we have that  $(\xi_{t,x}(s),p_{t,x}(s))\in\mathbf{K}^*_{x,\lambda_0}$ for all $s\in[0,t]$, where the compact set $\mathbf{K}^*_{x,\lambda_0}$ is defined in \eqref{eq:K}. Therefore, being solutions of the Hamiltonian system
\begin{equation*}
\begin{cases}
 \dot{\xi}_{t,x}(s)=H_p(\xi_{t,x}(s),p_{t,x}(s))
 \\
 \dot{p}_{t,x}(s)=-H_x(\xi_{t,x}(s),p_{t,x}(s))
\end{cases}
\qquad s\in[0,t],
\end{equation*}
both $\{\dot{\xi}_{t,x}(\cdot)\}_{t\in(0,t_0]}$ and $\{\dot{p}_{t,x}(\cdot)\}_{t\in(0,t_0]}$ are uniformly bounded. Consequently, 
$$
\ddot{\xi}_{t,x}(s)=H_{px}(\xi_{t,x}(s),p_{t,x}(s))\dot{\xi}_{t,x}(s)+H_{pp}(\xi_{t,x}(s),p_{t,x}(s))\dot{p}_{t,x}(s) \qquad(s\in[0,t])
$$
is also bounded, uniformly for  $t\in(0,t_0]$.
\end{proof}

\begin{Pro}\label{Main_lemma_g_c}
Let $L$ be a Tonelli Lagrangian. Let $t_0\in(0,1]$ be given by Lemma \ref{maximizer_singular}. For any fixed $x\in\R^n$, let $\mathbf{y}:[0,t_0]\to\R^n$ be the curve constructed in Lemma \ref{maximizer_singular}. Then
\begin{enumerate}[\rm (a)]
  \item $\mathbf{y}$ is Lipschitz on $[0,t_0]$. 
\end{enumerate}
Moreover, for any $t\in(0,t_0]$, let $\xi_{t,x}\in\Gamma^t_{x,\mathbf{y}(t)}$ be a minimizer for $A_t(x,\mathbf{y}(t))$. Then the following properties hold true:
\begin{enumerate}[\rm (a)]\setliststart{2}
  \item The right derivative $\dot{\mathbf{y}}^+(0)$ exists and
  \begin{equation}\label{eq:limit_1}
  	  \dot{\mathbf{y}}^+(0)=\lim_{t\to0^+}\dot{\xi}_{t,x}(t)=H_p(x,p_x)
  \end{equation}
  where $p_x$ is the unique element of $D^+u(x)$ such that
  $$
  H(x,p)\geqslant H(x,p_x),\quad \forall p\in D^+u(x).
  $$
  \item The arc $\mathbf{p}(t):=L_v(\xi_{t,x}(t),\dot{\xi}_{t,x}(t))$ is continuos on $(0,t_0]$ and $\lim_{t\to0^+}\mathbf{p}(t)=p_x$.  
  \item There exist $0<\rho\leqslant t_0$ and constants $C_1,C_2>0$ such that
    \begin{equation}\label{eq:Holder_esitmate_p}
    	H(\mathbf{y}(t),\mathbf{p}(t))\leqslant H(x,p_x)+C_1t-C_2|\mathbf{p}(t)-p_x|^2,\quad \forall t\in(0,\rho].
    \end{equation}
\end{enumerate}
\end{Pro}

\begin{proof}
Having fixed $x\in\R^n$, we shall abbreviate $\xi_{t,x}=\xi_{t}$. Let $0<t,s\leqslant t_0$ and let $\xi_t\in\Gamma^t_{x,\mathbf{y}(t)}$, $\xi_s\in\Gamma^s_{x,\mathbf{y}(s)}$ and $\eta\in\Gamma^t_{x,\mathbf{y}(s)}$ be minimizers for $A_t(x,\mathbf{y}(t))$, $A_s(x,\mathbf{y}(s))$, and $A_t(x,\mathbf{y}(s))$ respectively. Setting $p_t=L_v(\xi_t(t),\dot{\xi}_t(t))$, $p_s=L_v(\xi_s(s),\dot{\xi}_s(s))$, and $p=L_v(\eta(t),\dot{\eta}(t))$ we have that 
\begin{align*}
	\frac{C_2}t|\mathbf{y}(t)-\mathbf{y}(s)|^2
	\\\leqslant&\langle p_t-p,\mathbf{y}(t)-\mathbf{y}(s)\rangle=\langle p_t-p_s,\mathbf{y}(t)-\mathbf{y}(s)\rangle+\langle p_s-p,\mathbf{y}(t)-\mathbf{y}(s)\rangle\\
	\leqslant& C_1|\mathbf{y}(t)-\mathbf{y}(s)|^2+\langle p_s-p,\mathbf{y}(t)-\mathbf{y}(s)\rangle,
\end{align*}
where we have used the notation of the proof of Lemma~\ref{maximizer_singular}. By Proposition \ref{C11_A_t}, the function $(t,y)\mapsto A_t(x,y)$ is locally $C^{1,1}$ in the set $\{(t,y)\in\R\times\R^n: 0<t<t_0, |y-x|<\lambda t\}$. Moreover,  Proposition~\ref{semiconcave_A_t} together with Proposition~\ref{convexity_A_t} ensures that
$$
|p_s-p|\leqslant\frac{C_3}t|s-t|
$$
for some constant $C_3>0$. Therefore
$$
\left(\frac{C_2}t-C_1\right)|\mathbf{y}(t)-\mathbf{y}(s)|^2\leqslant\frac{C_3}t|s-t||\mathbf{y}(t)-\mathbf{y}(s)|.
$$
Recalling \eqref{eq:t_0} we have that $C_2/t-C_1>0$ for all $0<t\leqslant t_0$. Thus
$$
|\mathbf{y}(t)-\mathbf{y}(s)|\leqslant\frac{C_3}{C_2-C_1t_0}|t-s|,
$$
and this proves (a). 

Now we turn to the proof of (b). Since $\{\dot{\xi}_{t}(\cdot)\}_{t\in(0,t_0]}$ are equi-Lipschitz by Lemma \ref{uni_Lip_1}, for any sequence $t_k\to0^+$ such that $v_{k}:=(\xi_{t_k}(t_k)-x)/t_k$ converges, we obtain
\begin{equation}\label{eq:limit_velociy}
	\begin{split}
		\left|\frac{\xi_{t_k}(t_k)-x}{t_k}-\dot{\xi}_{t_k}(t_k)\right|
&\leqslant\frac 1{t_k}\int^{t_k}_{0}|\dot{\xi}_{t_k}(s)-\dot{\xi}_{t_k}(t_k)|\ ds\\
&\leqslant\frac C{t_k}\int^{t_k}_{0}(t_k-s)\ ds=\frac C2t_k.
	\end{split}
\end{equation}
This implies that
$$
v_{0}:=\lim_{k\to\infty}v_{k}=\lim_{k\to\infty}\dot{\xi}_{t_k}(t_k).
$$
By the semiconcavity of $u$, for any $p\in D^+u(x)$, we have
\begin{align*}
u(x)&\leqslant u(\mathbf{y}(t_k))+\langle L_v(\mathbf{y}(t_k),\dot{\xi}_{t_k}(t_k)), x-\mathbf{y}(t_k)\rangle+\frac C2|x-\mathbf{y}(t_k)|^2\\
    &\leqslant u(x)+\langle p,\mathbf{y}(t_k)-x\rangle+\langle L_v(\mathbf{y}(t_k),\dot{\xi}_{t_k}(t_k)), x-\mathbf{y}(t_k)\rangle+C|x-\mathbf{y}(t_k)|^2.
\end{align*}
Then, recalling that $\xi_{t_k}(t_k)=\mathbf{y}(t_k)$ we have
\begin{equation}\label{basic_inequalty_1}
\langle p-L_v(\mathbf{y}(t_k),\dot{\xi}_{t_k}(t_k)), v_{k}\rangle+t_kC|v_{k}|^2\geqslant 0,\quad \forall p\in D^+u(x).
\end{equation}
Taking the limit in \eqref{basic_inequalty_1} as $k\to\infty$ we obtain
\begin{equation}
\label{eq:key_propagation}
\langle p,v_{0}\rangle\geqslant\langle L_v(x,v_{0}),v_{0}\rangle=\langle p_x,v_{0}\rangle,\quad \forall p\in D^+u(x),
\end{equation}
where $p_x:=L_v(x,v_{0})\in D^+u(x)$ by the upper semicontinuity of  $x\rightsquigarrow D^+u(x)$. So,
\begin{equation}\label{optimization_problem_1}
H(x,p)\geqslant\langle L_v(x,v_{0}),v_{0}\rangle-L(x,v_{0})=H(x,p_x),\quad \forall p\in D^+u(x),
\end{equation}
and $p_x$ is the unique minimum point of $H(x,\cdot)$ on $D^+u(x)$. The uniqueness of $p_x$ implies the uniqueness of $v_{0}$ since $L_v(x,\cdot)$ is injective. This leads to the assertion that
$$
v_{0}=\lim_{t\to0^+}\frac{\xi_{t}(t)-x}{t}=\lim_{t\to0^+}\dot{\xi}_{t}(t)
$$
and, together with \eqref{optimization_problem_1}, implies \eqref{eq:limit_1}. This completes the proof of (b).

The conclusion (c) is a straight consequence of (a), (b) and the locally $C^{1,1}$ regularity property of the function $(t,y)\mapsto A_t(x,y)$.

Finally, we turn to prove (d). First, using Tailor's expansion, we have that
\begin{align*}
	&H(x,p_x)-H(\mathbf{y}(s),\mathbf{p}(s))\\
	=&H_x(\mathbf{y}(s),\mathbf{p}(s))(x-\mathbf{y}(s))+H_p(\mathbf{y}(s),\mathbf{p}(s))(p_x-\mathbf{p}(s))\\
	&+\frac 12\langle(H_{xp}(\mathbf{y}(s),\mathbf{p}(s))+H_{px}(\mathbf{y}(s),\mathbf{p}(s)))(p_x-\mathbf{p}(s)),(x-\mathbf{y}(s))\rangle\\
	&+\frac 12\langle H_{xx}(\mathbf{y}(s),\mathbf{p}(s))(x-\mathbf{y}(s)),(x-\mathbf{y}(s))\rangle\\
	&+\frac 12\langle H_{pp}(\mathbf{y}(s),\mathbf{p}(s))(p_x-\mathbf{p}(s)),(p_x-\mathbf{p}(s))\rangle+o(|\mathbf{y}(s)-x|^2+|\mathbf{p}(s)-p_x|^2).
\end{align*}
Thus, by (a), (b), (c) and our assumptions on $H$, there exist $\rho>0$ such that, for $s\in(0,\rho]$, we have
\begin{align*}
	&H(x,p_x)-H(\mathbf{y}(s),\mathbf{p}(s))\\
	\geqslant &-C_1s+\langle\dot{\xi}_s(s),p_x-\mathbf{p}(s)\rangle-C_{\varepsilon}s^2-\varepsilon|\mathbf{p}(s)-p_x|^2+C_2|\mathbf{p}(s)-p_x|^2
\end{align*}
Taking $\varepsilon>0$ small enough, we have
$$
H(x,p_x)-H(\mathbf{y}(s),\mathbf{p}(s))\geqslant-C_3s+\langle\dot{\xi}_s(s),p_x-\mathbf{p}(s)\rangle+C_4|\mathbf{p}(s)-p_x|^2.
$$
In view of \eqref{eq:limit_velociy}, we have
$$
H(x,p_x)-H(\mathbf{y}(s),\mathbf{p}(s))\geqslant-C_5s+\left\langle \frac{\mathbf{y}(s)-x}{s},p_x-\mathbf{p}(s)\right\rangle+C_4|\mathbf{p}(s)-p_x|^2.
$$
Therefore, by the semiconcavity of $u$, we obtain
$$
H(x,p_x)-H(\mathbf{y}(s),\mathbf{p}(s))\geqslant-C_6s+C_4|\mathbf{p}(s)-p_x|^2,
$$
which completes the proof of (d).
\end{proof}

\begin{Rem} Observe that \eqref{eq:key_propagation}, that is,
\begin{equation*} 
\langle p-p_x,v_0\rangle\geqslant0,\quad\forall p\in D^+u(x),
\end{equation*}
is exactly the key  condition for  propagation of singularities in \cite{Albano-Cannarsa} and \cite{Cannarsa-Yu}.
\end{Rem}

\begin{The}\label{sing_arc_gen_char}
Let $L$ be a Tonelli Lagrangian and let $H$ be the associated Hamiltonian. Suppose $u:\R^n\to\R$ is a Lipschitz continuous semiconcave viscosity solution of \eqref{eq:HJ-general} and $x\in\Singu$. Then the singular arc $\mathbf{y}:[0,t_0]\to\R^n$ defined in Lemma \ref{maximizer_singular} 
is a generalized characteristic  
and satisfies
\begin{equation}\label{eq:sing_gen_char}
\dot{\mathbf{y}}(\tau)\in\mbox{\rm co}\, H_p(\mathbf{y}(\tau),D^+u(\mathbf{y}(\tau))),\quad\text{a.e.}\ \tau\in[0,t_0].
\end{equation}
Moreover, 
\begin{equation}\label{eq:strong_gen_char_1}
\dot{\mathbf{y}}^+(0)=H_p(x,p_0),
\end{equation}
where $p_0$ is the unique element of minimal energy:
$$
H(x,p)\geqslant H(x,p_0),\quad \forall p\in D^+u(x).
$$
\end{The}

\begin{proof}
The conclusion can be derived directly from Lemma \ref{maximizer_singular} and Proposition \ref{Main_lemma_g_c}  except for \eqref{eq:sing_gen_char}.  For the proof of \eqref{eq:sing_gen_char}, see Appendix \ref{App.B}.
\end{proof}

To study the genuine propagation of singularities along generalized characteristics, we have to check that the singular arc $\mathbf{y}(t)$ in Lemma \ref{maximizer_singular} does not keep constant locally. As we show below,  the following condition can be useful for this purpose:
\begin{equation}\label{eq:alternative}
	D_yA_t(x,x)\not\in D^+u(x),\quad\ \text{for all}\ t\in(0,t_0].
\end{equation}

\begin{Pro}\label{keep_constant}
Let $\mathbf{y}:[0,t_0]\to\R^n$ be the singular generalized characteristic in Theorem \ref{sing_arc_gen_char}, and let $t\in(0,t_0]$. Then $\mathbf{y}(t)=x$ if and only if $D_yA_t(x,x)\in D^+u(x)$. Consequently, if \eqref{eq:alternative} holds, then $\mathbf{y}(t)\neq x$ for every $t\in(0,t_0]$.
\end{Pro}

\begin{proof}
Let $p\in D^+u(x)$, $p'=D_yA_t(x,x)$, and let $t\in(0,t_0]$. Recalling that $\mathbf{y}(t)$ is the unique maximizer of $\phi^x_t$ we have 
\begin{equation*}
\begin{split}
0\leqslant\phi^x_t(\mathbf{y}(t))-\phi^x_t(x)=&[u(\mathbf{y}(t))-u(x)-\langle p,\mathbf{y}(t)-x\rangle-\frac {C_1}2|\mathbf{y}(t)-x|^2]\\
&-[A_t(x,\mathbf{y}(t))-A_t(x,x)-\langle p',\mathbf{y}(t)-x\rangle-\frac {C_2}{2t}|\mathbf{y}(t)-x|^2]\\
&+\langle p-p',\mathbf{y}(t)-x\rangle+\frac 12\Big(C_1-\frac{C_2}t\Big)|\mathbf{y}(t)-x|^2,\\
\leqslant&\langle p-p',\mathbf{y}(t)-x\rangle+\frac 12\Big(C_1-\frac{C_2}t\Big)|\mathbf{y}(t)-x|^2,
\end{split}
\end{equation*}
where---like in the proof of  Lemma~\ref{maximizer_singular}---$C_1>0$ is a semiconcavity constant for $u$ on $\R^n$ and $C_2>0$  a convexity constant  for $A_t(x,\cdot)$ on
$B(x,(1+\lambda_0) t)$. So, 
\begin{equation*}
0\leqslant|\mathbf{y}(t)-x|\leqslant\frac{2|p-p'|}{C_2/t-C_1}.
\end{equation*}
If $D_yA_t(x,x)\in D^+u(x)$, then taking $p=p'$ in the above inequality  yields $\mathbf{y}(t)=x$. Conversely, if $\mathbf{y}(t)=x$, then the nonsmooth Fermat rule yields $0\in D^+u(x)-D_yA_t(x,x)$ which completes the proof.
\end{proof}

Another condition that ensures the genuine propagation of singularities is related to the notion of critical point.
\begin{defn}\label{defn:critical_point}
We say that $x\in\R^n$  is a {\em critical point} of  a viscosity solution  $u$ of \eqref{eq:HJ-general} if
$
0\in \text{co}\,H_p(x,D^+u(x)),
$
and a {\em strong critical point} of $u$  if
$
0\in H_p(x,D^+u(x)).
$
\end{defn}

\begin{Rem}
For a  mechanical Lagrangian of the form 
\begin{equation}
\label{eq:mechanical_systems}
L(x,v)=\frac 12\langle A(x)v,v\rangle-V(x),
\end{equation}
with $\langle A(x)\cdot,\cdot\rangle$ the matrix associated with a Riemannian metric in $\R^n$ and $V$ a smooth potential, $x$ is a critical point of a semiconcave solution $u$ of the corresponding Hamilton-Jacobi equation 
\begin{equation*}
\frac 12\langle A(x)^{-1}Du,Du\rangle+V(x)=0
\end{equation*}
if and only if $0\in D^+u(x)$, i.e., $x$ is a critical point of $u$ in the sense of nonsmooth analysis.
\end{Rem}

It is already known the condition that $x$ is not a critical point is a key point to guarantee the genuine propagation of singularities along  generalized characteristics (see, for instance, \cite{Albano-Cannarsa}). 

\begin{Cor}\label{critical_pts_u_H_1}
Let $\mathbf{y}:[0,t_0]\to\R^n$ be the singular generalized characteristic in Theorem \ref{sing_arc_gen_char}. If $x$ is not a strong critical point of $u$ then there exists $t\in(0,t_0]$ such that $\mathbf{y}(s)\not=x$ for all $s\in(0,t]$ .
\end{Cor}
  
\begin{proof} It suffices to show that $0\in H_p(x,D^+u(x))$ whenever a sequence $t_k\to 0$ exists such that $\mathbf{y}(t_k)=x$ for all $k\in\N$. Indeed, denoting by $\xi_k\in\Gamma^{t_k}_{x,\mathbf{y}(t_k)}$  the unique minimizer of $A_{t_k}(x,\mathbf{y}(t_k))$,
as in the proof of Proposition~\ref{Main_lemma_g_c} 
we have that 
$$
\lim_{k\to\infty}\dot{\xi}_{k}(t_k)=\lim_{k\to\infty}\frac{\xi_{k}(t_k)-x}{t_k}=0
$$
because $\xi_{k}(t_k)=\mathbf{y}(t_k)$. Therefore the dual arc  $p_k(s)=L_v(\xi_k(s),\dot{\xi}_k(s))$ satisfies 
$$
\lim_{k\to\infty}H_p(\xi_{k}(t_k),p_{k}(t_k))=\lim_{k\to\infty}\dot{\xi}_{k}(t_k)=0.
$$
Now, since $p_{k}(t_k)=D_yA_{t_k}(x,\mathbf{y}(t_k))\in D^+u(\mathbf{y}(t_k))$ by the nonsmooth Fermat rule,  the upper semicontinuity of $z\rightsquigarrow H_p(x,D^+u(z))$  yields $0\in H_p(x,D^+u(x))$.
\end{proof}

The above results on the propagation of singularities along generalized characteristics leads to the following global propagation property.

\begin{The}\label{global_propagation}
	Let $L$ be a Tonelli Lagrangian and let $H$ be the associated Hamiltonian. Suppose $u:\R^n\to\R$ is a Lipschitz continuous semiconcave viscosity solution of \eqref{eq:HJ-general}.
 If $x\in\Singu$, then there exists a generalized characteristic $\mathbf{x}:[0,+\infty)\to\R^n$ such that $\mathbf{x}(0)=x$ and $\mathbf{x}(s)\in\Singu$ for all $s\in [0,+\infty)$.
\end{The}

For geodesic systems, global propagation results were obtained in \cite{Albano2014_1}, \cite{Albano2014_2}, and \cite{acns} even on Riemannian manifolds. Theorem~\ref{global_propagation} above applies to mechanical systems with a Lagrangian $L$ of the form \eqref{eq:mechanical_systems}.

\begin{Cor}\label{global_propagation_2}
Let $L$ be the Tonelli Lagrangian in \eqref{eq:mechanical_systems} and let  $H$ be the associated Hamiltonian. Suppose that $A$ and $V$ are  bounded together with and all their derivatives up to the second order and let $u$ be a viscosity solution of the Hamilton-Jacobi equation \eqref{eq:HJ-general}.   If $x\in \Singu$, then there exists a unique generalized characteristic $\mathbf{x}:[0,+\infty)\to\R^n$ such that $\mathbf{x}(0)=x$ and $\mathbf{x}(s)\in\Singu$ for all $s\in [0,+\infty)$.
\end{Cor}
\begin{proof}
Observing that all the conditions {\rm (L1)-(L3)} are satisfied, the main part of the conclusion is an immediate consequence of Theorem~\ref{global_propagation}. The uniqueness of the generalized characteristic is a well-known consequence of the semiconcavity of $u$ (see, e.g., \cite{Cannarsa-Sinestrari}). 
\end{proof}

\subsection{Torus case}\label{se:torus}
In this section, we adapt our results to the flat $n$-torus $\T^n$. Given a Lagrangian $L$  on $\T^n\times\R^n$, of class $\mathcal C^2$, we lift $L$ to the universal covering space $\R^n\times\R^n$ and denote the lifted Lagrangian by $L$ as well. Then $L(x,v)$ is $\T^n$-periodic in $x$. For our regularity results, we suppose that $L$ satisfies conditions (L1) and (L2).

Let $Q:=(0,1]^n$ be a fundamental domain of the flat $n$-torus $\T^n$ lifted to the universal covering space $\R^n$. For each $x,x'\in\R^n$, we say that $x\sim x'$ if $x-x'\in\Z^n$ and we denote by $[x]$ the equivalence class of $x$. For any $x,y\in\R^n$ and $t>0$, we denote by $A_t([x],[y])$ the fundamental solution for $L$ on the torus, which is defined  as follows:
\begin{equation}\label{eq:fund_sol_torus}
	A_t([x],[y])=\inf_{x\in[x],y\in[y]}A_t(x,y),\qquad \forall [x],[y]\in\T^n.
\end{equation}
Similarly, we denote by $H$ be the associated $\T^n$-periodic Hamiltonian. Let $u$ be a $\T^n$-periodic viscosity solution of the Hamilton-Jacobi equation
$$
H(x,Du(x))=0\qquad(x\in\R^n).
$$
The weak KAM solution on the torus, associated with $u$, has the form
\begin{equation}
	u([x])=u(x)\quad\mbox{for any}\quad x\in[x], x\in\R^n, [x]\in\T^n.
\end{equation}

\begin{Lem}\label{fund_sol_torus}
Let $L$ be a $\T^n$-periodic Tonelli Lagrangian satisfying {\rm (L1)} and {\rm (L2)}. There exists $t_0>0$ such that for any $0<t\leqslant t_0$, if $d([x],[y])<t$, then $x,y\in Q$ exist such that $x\in [x]$ and $y\in [y]$, $|x-y|<t$, and $A_t([x],[y])=A_t(x,y)$.
\end{Lem}

\begin{proof}
For any $t>0$ and $x,y\in\R^n$, let $\sigma(s)=x+\frac st(y-x)$ for all $s\in[0,t]$. Then
\begin{equation}
\label{eq:torus}
A_t(x,y)\leqslant \int^t_0L(\sigma(s),\dot{\sigma}(s)) \leqslant t\kappa_1(|y-x|/t),
\end{equation}
where $\kappa_1(r)=\max_{z\in\T^n, |v|\leqslant r}L(z,v)$. On the other hand, for fixed $k\geqslant0$, we have
\begin{align}\nonumber
	A_t(x,y)=&\inf_{\xi\in\Gamma^t_{x,y}}\int^t_0L(\xi(s),\dot{\xi}(s))\ ds\geqslant\inf_{\xi\in\Gamma^t_{x,y}}\int^t_0\theta(|\dot{\xi}(s)|)-c_0\ ds\\\label{eq:torus2}
	\geqslant&\inf_{\xi\in\Gamma^t_{x,y}}\int^t_0\big(k|\dot{\xi}(s)|-\theta^*(k)-c_0\big)  ds\\\nonumber
	=& \,k|y-x|-t(\theta^*(k)+c_0)=C_1|y-x|-tC_2.
\end{align}
Let $x,y\in Q$ be such that $x\in [x]$ and $y\in [y]$, and let  $y'\not\in Q$ be also in $[y]$. Setting 
$$
t_0=\frac{C_1\sqrt{n}}{\kappa_1(1)+C_2+C_1},
$$
for all $0<t\leqslant t_0$ and $|y-x|<t$, we have
$$
t\kappa_1(|y-x|/t)+tC_2+C_1|y-x|\leqslant C_1\sqrt{n}\leqslant C_1|y-y'|,
$$
since $|y-y'|\geqslant\sqrt{n}$. By the above inequality, \eqref{eq:torus}, and \eqref{eq:torus2}  we obtain
\begin{align*}
	A_t(x,y)\leqslant C_1(|y-y'|-|y-x|)-tC_2
	\leqslant C_1|y'-x|-tC_2\leqslant A_t(x,y').
\end{align*}
This leads to our conclusion.
\end{proof}

By appealing to Lemma \ref{fund_sol_torus} and the compactness of $\T^n$, one can adapt the proof of all the results of Appendix \ref{app_reg} and realize that these regularity properties of the fundamental solution hold in the torus case as well. Similarly, the global propagation result was obtained thanks to the local regularity properties and uniform estimates for fundamental solutions that are in turn consequences of our assumptions on the Lagrangian.
Since such estimates are valid for the torus in view of the compactness of $\T^n$, global propagation holds as well.

Now we can formulate our main result in the torus case.

\begin{The}\label{torus_global_propagation}
Let $L$ be the $\T^n$-periodic Tonelli Lagrangian, let $H$ be the associated Hamiltonian and let $u$ be a $\T^n$-periodic viscosity solution of the Hamilton-Jacobi equation $H(x,Du(x))=0$. If $x\in\Singu$, then there exists a generalized characteristic $\mathbf{x}:[0,+\infty)\to\R^n$ such that $\mathbf{x}(0)=x$ and $\mathbf{x}(s)\in\Singu$ for all $s\in [0,+\infty)$.
\end{The}
The extension of the above global propagation theorem to arbitrary manifolds requires rewriting the regularity results of Appendix \ref{app_reg} in local charts, which is much more technical than in Euclidean space. This will be the object of future studies. 

\appendix

\section{Uniform Lipschitz bound for minimizers}\label{App.A}

In this appendix we adapt to the present context a Lipschitz estimate for minimizers of the action functional that was obtained in \cite{Dal-Maso-Frankowska} (see also \cite{Ambrosio-Ascenzi-Buttazzo}). We give a detailed proof of this result for the readers' convenience. We assume that the Lagrangian $L:\R^n\times\R^n\to\R$  is a function of class $C^2$  that satisfies the following conditions:
\begin{enumerate}[(L1')]
\item {\em Convexity}: $L_{vv}(x,v)>0$ for all  $(x,v)\in\R^n\times \R^n$.
\item {\em Growth condition}: There exists a superlinear function $\theta:[0,+\infty)\to[0,+\infty)$ and a constant $c_0>0$ such that 
$$L(x,v)\geqslant\theta(|v|)-c_0\qquad\forall  (x,v)\in\R^n\times \R^n.$$
\item {\em Uniform bound}: There exists a nondecreasing function $K:[0,+\infty)\to[0,+\infty)$ such that
\begin{equation*}
L(x,v)\leqslant K(|v|)\qquad\forall  (x,v)\in\R^n\times \R^n.
\end{equation*}
\end{enumerate}
Observe that (L1')-(L3') are weaker than assumptions (L1)-(L3).

We define the {\em energy function}
$$
E(x,v)=\langle v,L_v(x,v)\rangle-L(x,v),\quad (x,v)\in\R^n\times\R^n.
$$

\begin{Pro}\label{Main_bound_Lem}
Let $t,R>0$ and  suppose $L$ satisfies condition {\rm (L1')-(L3')}. Given  any $x\in\R^n$ and $y\in \overline{B}(x,R)$, let $\xi\in\Gamma^t_{x,y}$ be a minimizer for $A_t(x,y)$. Then we have that
\begin{align}\label{eq:main_bound}
	\sup_{s\in[0,t]}|\dot{\xi}(s)|\leqslant \kappa(R/t),
\end{align}
where $\kappa:(0,\infty)\to(0,\infty)$ is nondecreasing.
Moreover, if $t\leqslant 1$, then   
\begin{equation}
\label{eq:main_bound_0}
\sup_{s\in[0,t]}|\xi(s)-x|\leqslant\kappa(R/t).
\end{equation}
\end{Pro}

\begin{proof}
Fix $t>0$, $R>0$, $x\in\R^n$, let $y\in \overline{B}(x,R)$, and let $\xi\in\Gamma^{t}_{x,y}$ be a minimizer for $A_t(x,y)$, i.e.,
	$$
	A_t(x,y)=\int^t_0L(\xi(s),\dot{\xi}(s))\ ds.
	$$
	Denoting by $\sigma\in\Gamma^t_{x,y}$ the straight line segment defined by $\sigma(s)=x+\frac st(y-x)$, $s\in[0,t]$, in view of  (L2') and (L3') we have that
	\begin{align*}
		&\int^t_0\theta(|\dot{\xi}(s)|)\ ds-c_0t\leqslant\int^t_0L(\xi(s),\dot{\xi}(s))\ ds\leqslant\int^t_0L(\sigma(s),\dot{\sigma}(s))\ ds\\
		=&\int^t_0L\Big(x+\frac st(y-x),\frac{y-x}t\Big)\ ds\leqslant tK(R/t).
	\end{align*}
Therefore
$$
\int^t_0\theta(|\dot{\xi}(s)|)\ ds\leqslant c_0t+tK(R/t)=tC_1(R/t)
$$
with $C_1(r)=K(r)+c_0$. Since $|\dot{\xi}(s)|\leqslant \theta(|\dot{\xi}(s)|)+\theta^*(1)$, where  $\theta^*$ is the convex conjugate of $\theta$ defined in \eqref{eq:convex_conj_theta},  we have that
$$
\int^t_0|\dot{\xi}(s)|\leqslant tC_2(R/t),
$$
with $C_2(r)=C_1(r)+\theta^*(1)$. Hence
	\begin{equation}\label{eq:bound_xi}
		|\xi(s)-x|\leqslant\int^s_0|\dot{\xi}(s)|\ ds\leqslant tC_2(R/t),\quad \forall s\in[0,t],
	\end{equation}
	and
	\begin{equation}\label{eq:inf_dot_xi}
			\inf_{s\in[0,t]}|\dot{\xi}(s)|\leqslant\frac1{t}\int^t_0|\dot{\xi}(s)|\ ds\leqslant C_2(R/t).
	\end{equation}
Now, define $l_{\xi}(s,\lambda)=L(\xi(s),\dot{\xi}(s)/\lambda)\lambda$ for all $s\in[0,t]$ and $\lambda>0$. Then we have
$$
\frac{d}{d\lambda}l_{\xi}(s,\lambda)\vert_{\lambda=1}=L(\xi(s),\dot{\xi}(s))-\langle\dot{\xi}(s),L_v(\xi(s),\dot{\xi}(s))\rangle=-E(\xi(s),\dot{\xi}(s)).
$$
Since the energy is constant along a minimizer, there exists a constant $c_{\xi}$ such that
$$
\frac{d}{d\lambda}l_{\xi}(s,\lambda)\vert_{\lambda=1}=c_{\xi},\quad\forall s\in[0,t].
$$
Moreover, a simple computation shows that $l_{\xi}(s,\lambda)$ is convex in $\lambda$. So, we have
$$
c_{\xi}\geqslant\sup_{\lambda<1}\frac{l_{\xi}(s,\lambda)-l_{\xi}(s,1)}{\lambda-1},\quad\forall s\in[0,t].
$$
Let us now take, in the above inequality, $\lambda=3/4$ and  $s_0\in[0,t]$ such that $|\dot{\xi}(s_0)|=\inf_{s\in[0,t]}|\dot{\xi}(s)|$.
Then,  by (L2'), (L3'), and \eqref{eq:inf_dot_xi}
we conclude that
\begin{equation}\label{eq:c_xi_lower_bound}
	\begin{split}
		c_{\xi}\geqslant& 4(l_{\xi}(s_0,1)-l_{\xi}(s_0,3/4))\geqslant 4(-c_0-l_{\xi}(s_0,3/4))\\
	=&-4c_0-3L\Big(\xi(s_0),\frac 43\dot{\xi}(s_0)\Big)\geqslant-4c_0-3K\Big(\frac43|\dot{\xi}(s_0)|\Big)\\\geqslant&-4c_0-3K\Big(\frac43C_2(R/t)\Big)=-C_3(R/t),
	\end{split}
\end{equation}	
where $C_3(r)=4c_0+3K(4C_2(r)/3)$. 

By the convexity of $\l_{\xi}(s,\cdot)$ we also have, for any $\varepsilon\in(0,1)$,
$$
c_{\xi}\leqslant\frac{l_{\xi}(s,2-\varepsilon)-l_{\xi}(s,1)}{1-\varepsilon}.
$$
In other words,
$$
(1-\varepsilon)c_{\xi}+\varepsilon l_{\xi}(s,1)\leqslant l_{\xi}(s,2-\varepsilon)-(1-\varepsilon)l_{\xi}(s,1).
$$
Moreover, again by convexity, we have that
$$
l_{\xi}(s,2-\varepsilon)=l_{\xi}\Big(s,\varepsilon\cdot\frac 1{\varepsilon}+(1-\varepsilon)\cdot 1\Big)\leqslant\varepsilon l_{\xi}\Big(s,\frac 1{\varepsilon}\Big)+(1-\varepsilon)l_{\xi}(s,1).
$$
Therefore
$$
(1-\varepsilon)c_{\xi}+\varepsilon l_{\xi}(s,1)\leqslant\varepsilon l_{\xi}\Big(s,\frac 1{\varepsilon}\Big),
$$
that is,
$$
(1-\varepsilon)c_{\xi}+\varepsilon L(\xi(s),\dot{\xi}(s))\leqslant L(\xi(s),\varepsilon\dot{\xi}(s)).
$$
Hence, combining \eqref{eq:c_xi_lower_bound} and condition (L2'), we obtain
$$
-(1-\varepsilon)C_3(R/t)+\varepsilon(\theta(|\dot{\xi}(s)|)-c_0)\leqslant L(\xi(s),\varepsilon\dot{\xi}(s)).
$$
Set $S_{\xi}=\{s\in[0,t]: |\dot{\xi}(s)|\geqslant2\}$ and $\varepsilon=\varepsilon(s)=1/|\dot{\xi}(s)|$ for $s\in S_{\xi}$. Then
$$
-C_3(R/t)+\frac 1{|\dot{\xi}(s)|}C_3(R/t)+\frac{\theta(|\dot{\xi}(s)|)-c_0}{|\dot{\xi}(s)|}\leqslant L\Big(\xi(s),\frac{\dot{\xi}(s)}{|\dot{\xi}(s)|}\Big)\leqslant K(1), \quad\forall s\in S_{\xi}.
$$
Thus
$$
\theta(|\dot{\xi}(s)|)\leqslant(K(1)+C_3(R/t))|\dot{\xi}(s)|+(c_0-C_3(R/t)), \quad\forall s\in S_{\xi}.
$$
Therefore, by the Young-Fenchel inequality we deduce that
$$
|\dot{\xi}(s)|\leqslant (c_0-C_3(R/t))+\theta^*(K(1)+C_3(R/t)+1):=C_4(R/t),\quad \forall s\in S_{\xi}.
$$
Consequently, 
\begin{equation}\label{eq:bound_velocity}
	\sup_{s\in[0,t]}|\dot{\xi}(s)|\leqslant\max\{2,C_4(R/t)\}:=C_5(R/t).
\end{equation}
The conclusion follows from \eqref{eq:bound_velocity} and \eqref{eq:bound_xi} taking $
\kappa(r)=\max\{C_5(r),C_2(r)\}
$.
\end{proof}

\begin{Cor}\label{Main_bound_Lem_p}
In Proposition~\ref{Main_bound_Lem},  assume the additional condition:
\begin{enumerate}[\mbox{\rm (L3'')}]
  \item There exists a nondecreasing function $K_1:[0,+\infty)\to[0,+\infty)$ such that 
\begin{equation*}
|L_v(x,v)|\leqslant K_1(|v|)\qquad\forall  (x,v)\in\R^n\times \R^n.
\end{equation*}
\end{enumerate}
Then the dual arc $p(\cdot)$ associated with $\xi(\cdot)$ satisfies
\begin{align}\label{eq:main_bound_p}
	\sup_{s\in[0,t]}|p(s)|\leqslant \kappa_1(R/t),
\end{align}
where $\kappa_1:(0,\infty)\to(0,\infty)$ is nondecreasing.
\end{Cor}

\begin{proof}
By (L3'') together with \eqref{eq:main_bound} and \eqref{eq:main_bound_p} follows from
	$$
\sup_{s\in[0,t]}|p(s)|=\sup_{s\in[0,t]}|L_v(\xi(s),\dot{\xi}(s))|\leqslant K_1(\kappa(R/t))=\kappa_1(R/t),
$$
where $\kappa_1(r)=K_1\circ\kappa(r)$.
\end{proof}

\section{Convexity and $C^{1,1}$ estimate of fundamental solutions}\label{app_reg}
Let $L$ be a Tonelli Lagrangian (which implies conditions (L1')-(L3') and (L3'') in Appendix \ref{App.A}). Then
we have the following fundamental bounds for the velocity of minimizers.

Fix $x\in\R^n$ and suppose $R>0$ and $L$ is a Tonelli Lagrangian.  For any $0<t\leqslant1$ and $y\in\overline{B}(x,R)$, let $\xi \in\Gamma^t_{x,y}$ be a minimizer for $A_t(x,y)$ and let $p$ be its dual arc. Then there exists a nondecreasing function $\kappa:(0,\infty)\to(0,\infty)$ such that
\begin{equation}
\label{eq:main_bound_not}
\sup_{s\in[0,t]}|\dot{\xi}(s)|\leqslant\kappa(R/t),\quad
	\sup_{s\in[0,t]}|p(s)|\leqslant\kappa(R/t),
\end{equation}
by Proposition \ref{Main_bound_Lem} and Corollary \ref{Main_bound_Lem_p}. Now, $x\in\R^n$ and $\lambda>0$  define  compact sets
\begin{equation}\label{eq:K}
	\begin{split}
	\mathbf{K}_{x,\lambda}&:=\overline{B}(x,\kappa(4\lambda))\times\overline{B}(0,\kappa(4\lambda))\subset\R^n\times\R^n,\\
	\mathbf{K}^*_{x,\lambda}&:=\overline{B}(x,\kappa(4\lambda))\times\overline{B}(0,\kappa(4\lambda))\subset\R^n\times(\R^n)^*.
	\end{split}
\end{equation}
The following  is one of the key technical points of this paper.

\begin{Pro}\label{compactness_condition}
Suppose $L$ is a Tonelli Lagrangian. Fix $x\in\R^n$, $\lambda>0$, $t\in (0,1)$, and $y\in B(x,\lambda t)$. Let $z\in\R^n$  and $h\in\R$ be such that
\begin{equation}
\label{eq:<<1}
|z|<\lambda t\qquad\mbox{and}\qquad -\frac t2<h<1-t.
\end{equation}
Then  any minimizer
$\xi\in\Gamma^{t+h}_{x,y+z}$ for $A_{t+h}(x,y+z)$ and  corresponding  dual arc $p$ satisfy the following inclusions
\begin{align*}
	\{(\xi(s),\dot{\xi}(s)):s\in[0,t+h]\}&\subset \mathbf{K}_{x,\lambda},\\
	\{(\xi(s),p(s)):s\in[0,t+h]\}&\subset \mathbf{K}^*_{x,\lambda}.
\end{align*}
\end{Pro}
\begin{proof}
Since  $t/2<t+h<1$ and $y+z\in B(x,2\lambda t)$ by \eqref{eq:<<1}, we can use \eqref{eq:main_bound_0} and \eqref{eq:main_bound_not} to obtain
\begin{equation*}
\sup_{s\in[0,t+h]}|\xi(s)-x|\leqslant \kappa\Big(\frac{2\lambda t}{t+h}\Big)\leqslant \kappa(4\lambda)
\end{equation*}
and 
\begin{equation*}
\sup_{s\in[0,t]}|\dot{\xi}(s)|
	\leqslant\kappa\Big(\frac{2\lambda t}{t+h}\Big)\leqslant \kappa(4\lambda).
\end{equation*}
Since a similar bound holds  true for $\sup_{s\in[0,t]}|p(s)|$, the conclusion follows.
\end{proof}
\begin{Rem}\label{re:<<<1}
For any $x\in\R^n$ and $y\in B(x, \lambda t)$, condition \eqref{eq:<<1} is satisfied when
\begin{equation}
\label{<<<1}
|h|<t/2 \qquad\mbox{and}\qquad |z|<\lambda t
\end{equation}
provided that $0<t<2/3$.
\end{Rem}

\subsection{Semiconcavity of the fundamental solution}
The role of semiconcavity in optimal control problems has been widely investigated, see \cite{Cannarsa-Sinestrari}. For the minimization problem in \eqref{fundamental_solution}, the local semiconcavity of $A_t(x,y)$ with respect to $y$ was proved in \cite{Bernard2008}. In this paper, we give a local semiconcavity result of the map $(t,y)\mapsto A_t(x,y)$. 

The following result is essentially known.

\begin{Pro}[Semiconcavity of the fundamental solution]\label{semiconcave_A_t}
Suppose $L$ is a Tonelli Lagrangian. Then for any $\lambda>0$ there exists a constant $C_\lambda>0$ such that for any $x\in\R^n$, $t\in(0,2/3)$, $y\in B(x,\lambda t)$, and $(h,z)\in\R\times\R^n$  satisfying $|h|<t/2$ and $|z|<\lambda t$ we have
\begin{equation}\label{eq:seminconcavity_A_t}
A_{t+h}(x,y+z)+A_{t-h}(x,y-z)-2A_t(x,y)\leqslant\frac {C_\lambda} t\big(|h|^2+|z|^2\big).
\end{equation}
Consequently, $(t,y)\mapsto A_t(x,y)$ is locally semiconcave in $(0,1)\times\R^n$, uniformly with respect to $x$.
\end{Pro}

\begin{Rem}
For the purposes of this paper, it suffices to assume $0<t<2/3$. In the general case $t>0$, a local semiconcavity result holds true for $A_t(x,y)$ in the same form as \eqref{eq:seminconcavity_A_t} with $C_\lambda$ depending on $t$.
\end{Rem}

\subsection{Main Regularity Lemma}
We begin with the following known properties.

\begin{Lem}\label{superdiff_fund_sol}
Suppose $L$ is a Tonelli Lagrangian on $\R^n$. For any $x,y\in\R^n$, $t>0$, let $\xi\in\Gamma^t_{x,y}$ be a minimizer for $A_t(x,y)$. Then $\xi\in C^2([0,t])$ is an extremal curve\footnote{An arc $\xi(s)$ is called an extremal curve if it satisfies the associated Euler-Lagrange equation.}, and the dual arc $p(s):=L_v(\xi(s),\dot{\xi}(s))$ satisfies the sensitivity relation 
\begin{equation}\label{eq:p_in_superdiff}
	p(s)\in D^+_yA_t(x,\xi(s)),\quad s\in[0,t].
\end{equation}
Moreover, $A_t(x,\cdot)$ is differentiable at $y$ if and only if there is a unique minimizer $\xi\in\Gamma^t_{x,y}$. In this case, we have
\begin{equation}\label{diff_of_fund_sol}
D_yA_t(x,y)=L_v(\xi(t),\dot{\xi}(t)).
\end{equation}
\end{Lem}

\begin{proof}
The sensitivity relation \eqref{eq:p_in_superdiff} is obtained in, for instance, \cite[Theorem 6.4.8]{Cannarsa-Sinestrari}, for a problem with initial cost. Here the proof is similar. The uniqueness of the minimizer and regularity are classical results.
\end{proof}

\begin{Lem}[Main Regularity Lemma]\label{regularity_fund_sol1}
Suppose $L$ is a Tonelli Lagrangian. Then for any $\lambda>0$ there exists $t_\lambda\in(0,1]$ and constants $C_\lambda, C'_\lambda, C''_\lambda>0$ such that, for all $t\in(0,t_\lambda)$,  $x\in \R^n$, $y_1,y_2\in B(x,\lambda t)$,  and any minimizer $\xi_i\,( i=1,2)$  for $A_t(x,y_i)$, we have
\begin{align}
\|\xi_2-\xi_1\|^2_{L^{\infty}(0,t)}\leqslant& \frac {C_\lambda} t|y_2-y_1|^2\label{eq:regulairty_xi}\\
\int^t_0|p_2-p_1|^2ds\leqslant&\frac {C'_\lambda}t|y_2-y_1|^2\label{eq:regulairty_p}\\
\int^t_0|\dot{\xi}_2-\dot{\xi}_1|^2ds\leqslant&\frac {C''_\lambda}t|y_2-y_1|^2,\label{eq:regulairty_dot_xi}
\end{align}
where $p_i$ denotes the dual arc of $\xi_i$.
\end{Lem}

\begin{proof} 
Since $L$ is a Tonelli Lagrangian, we have that  $\xi_i(s)\,( i=1,2)$ of class $C^2$ and, 
by Proposition \ref{compactness_condition} with $h=0=z$, it follows that 
\begin{equation*}
\sup_{s\in[0,t]}|\dot{\xi}_i(s)|\leqslant \kappa(4 \lambda),\quad\sup_{s\in[0,t]}|p_i(s)|\leqslant \kappa(4 \lambda),
\end{equation*}
where  $p_i(s)=L_v(\xi_i(s),\dot{\xi}_i(s))$. Moreover, the pair $(\xi_i(\cdot),p_i(\cdot))$ satisfies the  Hamiltonian system
$$
\begin{cases}
\dot{\xi_i}=H_p(\xi_i,p_i)\\
\dot{p_i}=-H_x(\xi_i,p_i)
\end{cases}
\quad \text{on}\ [0,t]
$$
with
\begin{equation*}
\xi_i(t)=y_i,\quad \xi_i(0)=x.
\end{equation*}
Furthermore,  owing to Lemma \ref{superdiff_fund_sol},
\begin{equation*}\label{eq:endpoint_condition_A_t}
p_i(s)\in D^+_yA_t(x,\xi_i(s)),\quad \forall s\in[0,t].
\end{equation*}
Therefore
$$
\frac 12\frac d{ds}|\xi_2-\xi_1|^2=\langle H_p(\xi_2,p_2)-H_p(\xi_1,p_1),\xi_2-\xi_1\rangle.
$$
Integrating over $[s,t]$, we conclude that 
\begin{equation*}\label{eq:differentiate_1_deduction}
\begin{split}
|\xi_2(t)-\xi_1(t)|^2-|\xi_2(s)-\xi_1(s)|^2\geqslant&-C_1\int^t_s\big(|\xi_2-\xi_1|^2+|p_2-p_1|\cdot|\xi_2-\xi_1|\big)d\tau\\
\geqslant&-C_1\int^t_s|p_2-p_1|^2d\tau-2C_{1}\int^t_s|\xi_2-\xi_1|^2d\tau,
\end{split}
\end{equation*}
where $C_1=C_1(\lambda)>0$ is an upper bound for $D^2H(x,p)$ on $\{(x,p)~:~|p|\leqslant \kappa(4\lambda)\}$. So, 
\begin{equation}\label{eq:A_t.2}
\|\xi_2-\xi_1\|^2_{L^{\infty}(0,t)}\leqslant |y_2-y_1|^2+C_1\int^t_s|p_2-p_1|^2d\tau+2C_{1}t\|\xi_2-\xi_1\|^2_{L^{\infty}(0,t)}.
\end{equation}
Now,
\begin{equation*}
\begin{split}
&\frac d{ds}\langle p_2-p_1,\xi_2-\xi_1\rangle\\
=&\langle p_2-p_1,H_p(\xi_2,p_2)-H_p(\xi_1,p_1)\rangle-\langle H_x(\xi_2,p_2)-H_x(\xi_1,p_1),\xi_2-\xi_1\rangle\\
=&\langle p_2-p_1,\widehat{H}_{px}(\xi_2-\xi_1)+\widehat{H}_{pp}(p_2-p_1)\rangle-\langle \widehat{H}_{xp}(p_2-p_0)+\widehat{H}_{xx}(\xi_2-\xi_1),\xi_2-\xi_1\rangle,
\end{split}
\end{equation*}
where
\begin{equation*}
\widehat{H}_{px}(s)=\int^1_0H_{px}\big(\lambda\xi_2(s)+(1-\lambda)\xi_1(s),\lambda p_2(s)+(1-\lambda)p_1(s)\big)\ d\lambda,
\end{equation*}
and $\widehat{H}_{pp}$, $\widehat{H}_{xp}$, $\widehat{H}_{xx}$ are defined in a similar way. Thus, owing to (L1)-(L3), and since $(\widehat{H}_{px}(s))^*=\widehat{H}_{xp}(s)$ where $(\widehat{H}_{px})^*$ stands for the adjoint matrix, we have
\begin{equation*}
\frac d{ds}\langle p_2-p_1,\xi_2-\xi_1\rangle\geqslant\nu|p_2-p_1|^2-C_2|\xi_2-\xi_1|^2
\end{equation*}
for some positive constants $\nu=\nu(\lambda)$ and $C_2=C_2(\lambda)$. Now, by \eqref{eq:seminconcavity_A_t},
\begin{equation*}
\begin{split}
&\nu\int^t_0|p_2-p_1|^2ds
\leqslant C_2\int^t_0|\xi_2-\xi_1|^2ds+\langle p_2(t)-p_1(t),\xi_2(t)-\xi_1(t)\rangle\\\leqslant& C_2\int^t_0|\xi_2-\xi_1|^2ds+\frac {C_3}t|y_2-y_1|^2.
\end{split}
\end{equation*}
Therefore,
\begin{equation}\label{eq:regularity_p}
\int^t_0|p_2-p_1|^2ds\leqslant\frac{C_2t}{\nu}\|\xi_2-\xi_1\|^2_{L^{\infty}(0,t)}+\frac{C_3}{\nu t}|y_2-y_1|^2.
\end{equation}
Combining \eqref{eq:A_t.2} and \eqref{eq:regularity_p}, we obtain
\begin{align*}
\|\xi_2-\xi_1\|^2_{L^{\infty}(0,t)}&\leqslant |y_2-y_1|^2+\frac{C_1 C_2t}{\nu}\|\xi_2-\xi_1\|^2_{L^{\infty}(0,t)}+\frac{C_1 C_3}{\nu t}|y_2-y_1|^2\\
&\ \ \ +2C_{1}t\|\xi_2-\xi_1\|^2_{L^{\infty}(0,t)},\\
&=\Big(\frac{ C_2}{\nu}+2\Big)C_{1}t\|\xi_2-\xi_1\|^2_{L^{\infty}(0,t)}+\Big(1+\frac{C_1 C_3}{\nu t}\Big)|y_2-y_1|^2
\end{align*}
Then, taking 
\begin{equation*}
t_\lambda=\min\Big\{1\,,\,\frac\nu{2C_1(C_2+2\nu)}\Big\},
\end{equation*}
for all $t\in (0,t_{\lambda})$ we conclude that
\begin{equation*}
\|\xi_2-\xi_1\|^2_{L^{\infty}(0,t)}\leqslant 2 \Big(1+\frac{C_1 C_3}{\nu t}\Big)|y_2-y_1|^2\leqslant \frac{C}{ t}|y_2-y_1|^2.
\end{equation*}
This proves \eqref{eq:regulairty_xi} and also \eqref{eq:regulairty_p} owing to  \eqref{eq:regularity_p}. Finally, observing that
\begin{align*}
&\int^t_0|\dot{\xi}_2-\dot{\xi}_1|^2ds=\int^t_0|H_p(\xi_2,p_2)-H_p(\xi_1,p_1)|^2ds\\
\leqslant&2\int^t_0|H_p(\xi_2,p_2)-H_p(\xi_2,p_1)|^2ds+\int^t_0|H_p(\xi_2,p_1)-H_p(\xi_1,p_1)|^2ds\\
\leqslant&2C_4\Big(\int^t_0|p_2-p_1|^2ds+\int^t_0|\xi_2-\xi_1|^2ds\Big)
\end{align*}
we obtain \eqref{eq:regulairty_dot_xi} by appealing to \eqref{eq:regulairty_xi} and \eqref{eq:regulairty_p}.
\end{proof}

\subsection{Convexity of the fundamental solution for small time}
For any $t>0$,  $x,y,z\in\R^n$, and any  $h\in [0,t)$,  let $\xi_+\in\Gamma^{t+ h}_{x,y+ z}$ and $\xi_-\in\Gamma^{t- h}_{x,y- z}$  be given. Define $\tilde{\xi}_\pm\in\Gamma^t_{x,y\pm z}$  by
\begin{equation}\label{def:ctilde_xi}
\tilde{\xi}_+(\tau)=\xi_+\Big(\frac{t+h}t\tau\Big),\quad \tilde{\xi}_-(\tau)=\xi_-\Big(\frac{t-h}t\tau\Big),\quad\tau\in[0,t].
\end{equation}
Obviously,
\begin{gather}
\tilde{\xi}_+(0)=\tilde{\xi}_-(0)=x,\quad \tilde{\xi}_+(t)=y+z,\quad \tilde{\xi}_-(t)=y-z,\notag\\
\frac{\tilde{\xi}_++\tilde{\xi}_-}2(0)=x,\quad\frac{\tilde{\xi}_++\tilde{\xi}_-}2(t)=y,\label{eq:tilde_sum_xi}\\
\frac{\tilde{\xi}_+-\tilde{\xi}_-}2(0)=0,\quad\frac{\tilde{\xi}_+-\tilde{\xi}_-}2(t)=z.\label{eq:tilde_diff_xi}
\end{gather}

\begin{Lem}\label{tilde_three_acrs}
Suppose $L$ is a Tonelli Lagrangian. For any $\lambda>0$  let  $t_\lambda>0$ be given by Lemma~\ref{regularity_fund_sol1} and define $t_\lambda'=\min\{t_\lambda,2/3\}$.
Then there exist constants $C_\lambda,C'_\lambda>0$ such that, for any $t\in (0,t_\lambda')$, any $x\in \R^n$, any $y\in B(x,\lambda t)$, any $(h,z)\in[0,t/2)\times B(0,\lambda t)$, and any pair of minimizers,  $\xi_\pm\in\Gamma^{t\pm h}_{x,y\pm z}$,  for
$A_{t\pm h}(x,y\pm z)$ we have 
\begin{align}
\|\tilde{\xi}_+-\tilde{\xi}_-\|^2_{L^{\infty}(0,t)}\leqslant& \frac {C_\lambda}t(h^2+|z|^2)\label{eq:L_infty_1},\\
\int^t_0|\dot{\tilde{\xi}}_+-\dot{\tilde{\xi}}_-|^2\ d\tau\leqslant& \frac {C'_\lambda}t(h^2+|z|^2),\label{eq:L_2}
\end{align}
where $\tilde{\xi}_{\pm}\in\Gamma^t_{x,y\pm z}$ are defined  in \eqref{def:ctilde_xi}.
\end{Lem}

\begin{proof}
In view of Proposition \ref{compactness_condition} and Remark~\ref{re:<<<1} we have that $\{(\xi_\pm(s),\dot\xi_\pm(s))\}_{s\in[0,t\pm h]}$ is contained in the convex compact set $\mathbf{K}_{x,\lambda}$. So, for any $\tau\in[0,t]$,
\begin{align*}
&|\tilde{\xi}_+(\tau)-\tilde{\xi}_-(\tau)|=\Big|\xi_+\Big(\frac{t+h}t\tau\Big)-\xi_-\Big(\frac{t-h}t\tau\Big)\Big|\\
\leqslant&\Big|\xi_+\Big(\frac{t+h}t\tau\Big)-\xi_+\Big(\frac{t-h}t\tau\Big)\Big|+\Big|\xi_+\Big(\frac{t-h}t\tau\Big)-\xi_-\Big(\frac{t-h}t\tau\Big)\Big|\\
\leqslant&C_1h+\Big|\xi_+\Big(\frac{t-h}t\tau\Big)-\xi_-\Big(\frac{t-h}t\tau\Big)\Big|\leqslant \kappa(4\lambda)h+\max_{s\in[0,t-h]}|\xi_+(s)-\xi_-(s)|.
\end{align*}
Since
\begin{align*}
|\xi_+(t-h)-\xi_-(t-h)|\leqslant&|\xi_+(t-h)-\xi_+(t+h)|+|\xi_+(t+h)-\xi_-(t-h)|\\
\leqslant& 2\big(\kappa(4\lambda)h+|z|\big),
\end{align*}
by \eqref{eq:regulairty_xi} applied to $\xi_+,\xi_-$ on $[0,t-h]$ we obtain
$$
\max_{s\in[0,t-h]}|\xi_+(s)-\xi_-(s)|^2\leqslant \frac{C_1}{t}(h^2+|z|^2),
$$
for some constant $C_1=C_1(\lambda)>0$. Similarly, we have
\begin{align*}
&\big|\dot{\tilde{\xi}}_+(\tau)-\dot{\tilde{\xi}}_-(\tau)\big|=\Big|\frac{t+h}t\dot{\xi}_+\Big(\frac{t+h}t\tau\Big)-\frac{t-h}t\dot{\xi}_-\Big(\frac{t-h}t\tau\Big)\Big|\\
\leqslant&\Big|\frac{t+h}t\dot{\xi}_+\Big(\frac{t+h}t\tau\Big)-\frac{t+h}t\dot{\xi}_-\Big(\frac{t-h}t\tau\Big)\Big|+\Big|\frac{t+h}t\dot{\xi}_-\Big(\frac{t-h}t\tau\Big)-\frac{t-h}t\dot{\xi}_-\Big(\frac{t-h}t\tau\Big)\Big|\\
\leqslant&\frac{t+h}t\Big|\dot{\xi}_+\Big(\frac{t+h}t\tau\Big)-\dot{\xi}_-\Big(\frac{t-h}t\tau\Big)\Big| +  \kappa(4\lambda)\frac{h}t.
\end{align*}
On the other hand,
\begin{align*}
&\Big|\dot{\xi}_+\Big(\frac{t+h}t\tau\Big)-\dot{\xi}_-\Big(\frac{t-h}t\tau\Big)\Big|\\
\leqslant&\Big|\dot{\xi}_+\Big(\frac{t+h}t\tau\Big)-\dot{\xi}_+\Big(\frac{t-h}t\tau\Big)\Big|+\Big|\dot{\xi}_+\Big(\frac{t-h}t\tau\Big)-\dot{\xi}_-\Big(\frac{t-h}t\tau\Big)\Big|\\
\leqslant& C_2\frac{h}t+\Big|\dot{\xi}_+\Big(\frac{t-h}t\tau\Big)-\dot{\xi}_-\Big(\frac{t-h}t\tau\Big)\Big|.
\end{align*}
Thus, for $t\in(0,t_\lambda')$  the bound in \eqref{eq:regulairty_dot_xi} yields
\begin{align*}
\int^t_0|\dot{\tilde{\xi}}_+(\tau)-\dot{\tilde{\xi}}_-(\tau)|^2\ d\tau\leqslant& C_3\frac{h^2}t+\frac{(t+h)^2}{t^2}\int^t_0\Big|\dot{\xi}_+\Big(\frac{t-h}t\tau\Big)-\dot{\xi}_-\Big(\frac{t-h}t\tau\Big)\Big|^2\ d\tau\\
=&C_3\frac{h^2}t+\frac{(t+h)^2}{t^2}\cdot\frac t{t-h}\int^{t-h}_0|\dot{\xi}_+(s)-\dot{\xi}_-(s)|^2\ ds\\
\leqslant&C_3\frac{h^2}t+\frac{(t+h)^2}{t^2}\cdot\frac t{t-h}\cdot \frac{C_4}{t-h}(h^2+|z|^2).
\end{align*}
This leads to our conclusion.
\end{proof}

\begin{Pro}\label{convexity_A_t}
Suppose $L$ is a Tonelli Lagrangian and, for any $\lambda>0$,  let  $t_\lambda'>0$ be the number given by Lemma~\ref{tilde_three_acrs}. 
Then, for any $x\in\R^n$, the function $(t,y)\mapsto A_t(x,y)$ is semiconvex on the  cone
\begin{equation}
\label{cone}
S_\lambda(x,t_\lambda'):=\big\{(t,y)\in\R\times\R^n~:~0<t< t_\lambda',\; |y-x|<\lambda t\big\}\,,
\end{equation}
and there exists a constant $C''_\lambda>0$ such that for all $(t,y)\in S_\lambda(x,t_\lambda')$,  all $h\in[0,t/2)$, and  all $z\in  B(0,\lambda t)$ we have that
\begin{equation}\label{eq:semiconvexity}
A_{t+h}(x,y+z)+A_{t-h}(x,y-z)-2A_t(x,y)\geqslant - \frac{C''_ \lambda}{t}(h^2+|z|^2).
\end{equation}
Moreover, there exists $t''_\lambda\in(0,t_\lambda']$ and  $C'''_{\lambda}>0$ such that for all $t\in(0,t''_\lambda]$ the function $A_t(x,\cdot)$ is uniformly convex on $B(x,\lambda t)$ and   for all $y\in B(x,\lambda t)$ and   $z\in  B(0,\lambda t)$ we have that
\begin{equation}\label{eq:convexity_local}
A_{t}(x,y+z)+A_{t}(x,y-z)-2A_t(x,y)\geqslant \frac{C'''_{\lambda}}{t}|z|^2.
\end{equation}
\end{Pro}

\begin{proof}
Let  $x\in\R^n$ and fix $(t,y)\in S_\lambda(x,t_\lambda')$, $h\in[0,t/2)$, and  $z\in  B(0,\lambda t)$. Let $\xi_+\in\Gamma^{t+h}_{x,y+z}$ and $\xi_-\in\Gamma^{t-h}_{x,y-z}$ be  minimizers for $A_{t+h}(x,y+z)$ and $A_{t-h}(x,y-z)$ respectively, and define $\tilde{\xi}_{\pm}$ as in \eqref{def:ctilde_xi}. In view of  Proposition \ref{compactness_condition} and Remark~\ref{re:<<<1} we have that $\{(\xi_\pm(s),\dot\xi_\pm(s))\}_{s\in[0,t\pm h]}$ and $\{(\tilde\xi_\pm(s),\dot{\tilde\xi}_\pm(s))\}_{s\in[0,t]}$ are all contained  in the convex compact set $\mathbf{K}_{x,\lambda}$. Moreover
\begin{equation*}
\begin{split}
&A_{t+h}(x,y+z)+A_{t-h}(x,y-z)\\
=&\int^{t+h}_0L(\xi_+(s),\dot{\xi}_+(s))\ ds+\int^{t-h}_0L(\xi_-(s),\dot{\xi}_-(s))\ ds\\
=&\frac{t+h}t\int^t_0L\Big(\xi_+\Big(\frac{t+h}t\tau\Big),\dot{\xi}_+\Big(\frac{t+h}t\tau\Big)\Big)\ d\tau\\
&+\frac{t-h}t\int^t_0L\Big(\xi_-\Big(\frac{t-h}t\tau\Big),\dot{\xi}_-\Big(\frac{t-h}t\tau\Big)\Big)\ d\tau\\
=&\frac{t+h}t\int^t_0L\Big(\tilde{\xi}_+(\tau),\frac t{t+h}\dot{\tilde{\xi}}_+(\tau)\Big)\ d\tau+\frac{t-h}t\int^t_0L\Big(\tilde{\xi}_-(\tau),\frac t{t-h}\dot{\tilde{\xi}}_-(\tau)\Big)\ d\tau\\
=&\int^t_0L\Big(\tilde{\xi}_+(\tau),\frac t{t+h}\dot{\tilde{\xi}}_+(\tau)\Big)\ d\tau+\int^t_0L\Big(\tilde{\xi}_-(\tau),\frac t{t-h}\dot{\tilde{\xi}}_-(\tau)\Big)\ d\tau+I_1,
\end{split}
\end{equation*}
where
$$
I_1=\frac ht\int^t_0\Big\{L\Big(\tilde{\xi}_+(\tau),\frac t{t+h}\dot{\tilde{\xi}}_+(\tau)\Big)-L\Big(\tilde{\xi}_-(\tau),\frac t{t-h}\dot{\tilde{\xi}}_-(\tau)\Big)\Big\}\ d\tau.
$$
Set
\begin{equation*}
\begin{split}
I_2=&\int^t_0L(\tilde{\xi}_+(\tau),\dot{\tilde{\xi}}_+(\tau))\ d\tau+\int^t_0L(\tilde{\xi}_-(\tau),\dot{\tilde{\xi}}_-(\tau))\ d\tau-2A_t(x,y),\\
I_3=&\int^t_0L\Big(\tilde{\xi}_+(\tau),\frac t{t+h}\dot{\tilde{\xi}}_+(\tau)\Big)\ d\tau-\int^t_0L(\tilde{\xi}_+(\tau),\dot{\tilde{\xi}}_+(\tau))\ d\tau\\
&+\int^t_0L\Big(\tilde{\xi}_-(\tau),\frac t{t-h}\dot{\tilde{\xi}}_-(\tau)\Big)\ d\tau-\int^t_0L(\tilde{\xi}_-(\tau),\dot{\tilde{\xi}}_-(\tau))\ d\tau.
\end{split}
\end{equation*}
Then 
$$
A_{t+h}(x,y+z)+A_{t-h}(x,y-z)-2A_t(x,y)=I_1+I_2+I_3.
$$
Now we turn to the estimates of $I_1$, $I_2$ and $I_3$.
\medskip

\noindent {\bf Estimate of $I_1$}: Let $C_0=C_0(\lambda)>0$ be an upper bound for $|L_{v}|$ on $\mathbf{K}_{x,\lambda}$. Then
\begin{align*}
I_1=&\frac ht\int^t_0\Big\{L\Big(\tilde{\xi}_+,\frac t{t+h}\dot{\tilde{\xi}}_+\Big)-L\Big(\tilde{\xi}_+,\frac t{t-h}\dot{\tilde{\xi}}_+\Big)\Big\}\ d\tau\\
&+\frac ht\int^t_0\Big\{L\Big(\tilde{\xi}_+,\frac t{t-h}\dot{\tilde{\xi}}_+\Big)-L\Big(\tilde{\xi}_+,\frac t{t-h}\dot{\tilde{\xi}}_-\Big)\Big\}\ d\tau\\
&+\frac ht\int^t_0\Big\{L\Big(\tilde{\xi}_+,\frac t{t-h}\dot{\tilde{\xi}}_-\Big)-L\Big(\tilde{\xi}_-,\frac t{t-h}\dot{\tilde{\xi}}_-\Big)\Big\}\ d\tau\\
\geqslant&-C_0\frac {th^2}{t^2-h^2}-C_0\frac{h}{t-h}\int^t_0|\dot{\tilde{\xi}}_+-\dot{\tilde{\xi}}_-|d\tau-C_0\frac{h}t\int^t_0|\tilde{\xi}_+-\tilde{\xi}_-|d\tau
\end{align*}
By \eqref{eq:L_2} it is easy to check that
\begin{equation*}
\begin{split}
\int^t_0|\dot{\tilde{\xi}}_+-\dot{\tilde{\xi}}_-|d\tau\leqslant&\Big(\int^t_0|\dot{\tilde{\xi}}_+-\dot{\tilde{\xi}}_-|^2d\tau\Big)^{\frac 12}t^{\frac 12}
\leqslant\sqrt{C'_\lambda(h^2+|z|^2)}.
\end{split}
\end{equation*}
Similarly,  by \eqref{eq:L_infty_1},
$$\int^t_0|\tilde{\xi}_+-\tilde{\xi}_-|d\tau
\leqslant\sqrt{C_\lambda t(h^2+|z|^2)}.
$$
Thus, recalling that $h\in[0,t/2)$ and  $z\in  B(0,\lambda t)$, we conclude that 
\begin{eqnarray}
 \label{eq:I_1}
I_1&\geqslant&-\frac{2C_0}{t}h^2-\frac{2C_0}{t}h\sqrt{C'_\lambda(h^2+|z|^2)}-\frac{C_0}{t}h\sqrt{C_\lambda t(h^2+|z|^2)}
 \\\label{eq:I_11}
&\geqslant&
-\frac{C_1}{t}(h^2+|z|^2),
\end{eqnarray}
for some constant $C_1=C_1(\lambda)>0$.
\medskip

\noindent {\bf Estimate of $I_2$}: Let $\nu_\lambda:=\nu\big(\kappa(4\lambda)\big)>0$ be the lower bound for $L_{vv}$ on $\mathbf{K}_{x,\lambda}$ provided by assumption (L1). Then
  \begin{equation*}
  \begin{split}
  I_2=&\int^t_0L(\tilde{\xi}_+(\tau),\dot{\tilde{\xi}}_+(\tau))\ d\tau+\int^t_0L(\tilde{\xi}_-(\tau),\dot{\tilde{\xi}}_-(\tau))\ d\tau-2A_t(x,y),\\
  \geqslant&\int^t_0\Big\{L(\tilde{\xi}_+,\dot{\tilde{\xi}}_+)+L(\tilde{\xi}_-,\dot{\tilde{\xi}}_-)-2L\Big(\frac{\tilde{\xi}_++\tilde{\xi}_-}2,\frac{\dot{\tilde{\xi}}_++\dot{\tilde{\xi}}_-}2\Big)\Big\}\ d\tau\\
  =&\int^t_0\Big\{L\Big(\frac {\tilde{\xi}_++\tilde{\xi}_-}2,\dot{\tilde{\xi}}_+\Big)+L\Big(\frac {\tilde{\xi}_++\tilde{\xi}_-}2,\dot{\tilde{\xi}}_-\Big)-2L\Big(\frac {\tilde{\xi}_++\tilde{\xi}_-}2,\frac{\dot{\tilde{\xi}}_++\dot{\tilde{\xi}}_-}2\Big)\Big\}\ d\tau\\
&+\int^t_0\Big\{L(\tilde{\xi}_+,\dot{\tilde{\xi}}_+)-L\Big(\frac {\tilde{\xi}_++\tilde{\xi}_-}2,\dot{\tilde{\xi}}_+\Big)\Big\}d\tau+\int^t_0\Big\{L(\tilde{\xi}_-,\dot{\tilde{\xi}}_-)-L\Big(\frac {\tilde{\xi}_++\tilde{\xi}_-}2,\dot{\tilde{\xi}}_-\Big)\Big\}d\tau\\
\geqslant&\;\nu_\lambda\int^t_0\Big|\frac{\dot{\tilde{\xi}}_+-\dot{\tilde{\xi}}_-}2\Big|^2\ d\tau+\int^t_0\int^1_0\Big\langle L_x\Big(\lambda\tilde{\xi}_++(1-\lambda)\frac{\tilde{\xi}_++\tilde{\xi}_-}2,\dot{\tilde{\xi}}_+\Big),\frac{\tilde{\xi}_+-\tilde{\xi}_-}2\Big\rangle\ d\lambda d\tau\\
&+\int^t_0\int^1_0\Big\langle L_x\Big(\lambda\tilde{\xi}_-+(1-\lambda)\frac{\tilde{\xi}_++\tilde{\xi}_-}2,\dot{\tilde{\xi}}_-\Big),-\frac{\tilde{\xi}_+-\tilde{\xi}_-}2\Big\rangle\ d\lambda d\tau
  \end{split}
  \end{equation*}
   Setting 
$$
\widehat{L_x}(\lambda,\tau)=L_x\Big(\lambda\tilde{\xi}_++(1-\lambda)\frac{\tilde{\xi}_++\tilde{\xi}_-}2,\dot{\tilde{\xi}}_+\Big)-L_x\Big(\lambda\tilde{\xi}_-+(1-\lambda)\frac{\tilde{\xi}_++\tilde{\xi}_-}2,\dot{\tilde{\xi}}_-\Big),
$$
we have that
\begin{align}
\label{eq:I_2split}
\nonumber
I_2\geqslant&\; \nu_\lambda\int^t_0\Big|\frac{\dot{\tilde{\xi}}_+-\dot{\tilde{\xi}}_-}2\Big|^2\ d\tau+\int^t_0\int^1_0\big\langle \widehat{L_x},\frac{\tilde{\xi}_+-\tilde{\xi}_-}2\big\rangle\ d\lambda d\tau
\\
\geqslant&\;\nu_\lambda\int^t_0\Big|\frac{\dot{\tilde{\xi}}_+-\dot{\tilde{\xi}}_-}2\Big|^2\ d\tau-C_2\int^t_0 \big(\big|\tilde{\xi}_+-\tilde{\xi}_-\big|^2+\big|\dot{\tilde{\xi}}_+-\dot{\tilde{\xi}}_-\big|\cdot\big|\tilde{\xi}_+-\tilde{\xi}_-\big|\big)\ d\tau,
\end{align}
where $C_2=C_2(\lambda)>0$ is such that
\begin{equation}\label{eq:UB}
|L_{v}|, |L_{xx}|, |L_{xv}|, |L_{vv}|\leqslant C_2\qquad\mbox{on}\quad \mathbf{K}_{x,\lambda}.
\end{equation}

\medskip
\noindent {\bf Estimate of $I_3$}: As above, let $\nu_\lambda=\nu\big(\kappa(4\lambda)\big)>0$. Then
\begin{equation}\label{eq:I_3}
\begin{split}
I_3=&\int^t_0L\Big(\tilde{\xi}_+(\tau),\frac t{t+h}\dot{\tilde{\xi}}_+(\tau)\Big)\ d\tau-\int^t_0L(\tilde{\xi}_+(\tau),\dot{\tilde{\xi}}_+(\tau))\ d\tau\\
&+\int^t_0L\Big(\tilde{\xi}_-(\tau),\frac t{t-h}\dot{\tilde{\xi}}_-(\tau)\Big)\ d\tau-\int^t_0L(\tilde{\xi}_-(\tau),\dot{\tilde{\xi}}_-(\tau))\ d\tau\\
\geqslant&\int^t_0\Big\{\big\langle L_v(\tilde{\xi}_+,\dot{\tilde{\xi}}_+),-\frac h{t+h}\dot{\tilde{\xi}}_+\big\rangle+\frac {\nu_\lambda h^2}{|t+h|^2}|\dot{\tilde{\xi}}_+|^2 \Big\}\ d\tau\\
&+\int^t_0\Big\{\big\langle L_v(\tilde{\xi}_-,\dot{\tilde{\xi}}_-),\frac h{t-h}\dot{\tilde{\xi}}_-\big\rangle+ \frac {\nu_\lambda h^2}{|t-h|^2}|\dot{\tilde{\xi}}_-|^2\Big\} d\tau.
\end{split}
\end{equation}
Since $2\big(|\dot{\tilde{\xi}}_+|^2+|\dot{\tilde{\xi}}_-|^2\big)\geqslant \big|\dot{\tilde{\xi}}_+\pm \dot{\tilde{\xi}}_-|^2$ and
\begin{equation*}
\frac {1}{|t\pm h|^2}\geqslant \Big(\frac 2{3t}\Big)^2,
\end{equation*}
\eqref{eq:I_3} yields
\begin{equation}\label{eq:I_3true}
\begin{split}
I_3\geqslant&\;\nu_\lambda h^2\Big(\frac 2{3t}\Big)^2\Big\{\int^t_0\Big|\frac{\dot{\tilde{\xi}}_+-\dot{\tilde{\xi}}_-}2\Big|^2\ d\tau+\int^t_0\Big|\frac{\dot{\tilde{\xi}}_++\dot{\tilde{\xi}}_-}2\Big|^2\ d\tau\Big\}\\
+&\int^t_0\Big\{\big\langle L_v(\tilde{\xi}_-,\dot{\tilde{\xi}}_-),\frac h{t-h}\dot{\tilde{\xi}}_-\big\rangle-\big\langle L_v(\tilde{\xi}_+,\dot{\tilde{\xi}}_+),\frac h{t+h}\dot{\tilde{\xi}}_+\big\rangle\Big\} d\tau:=I_4+I_5
\end{split}
\end{equation}
Now, since $\tilde{\xi}_+-\tilde{\xi}_-$ is an arc connecting $0$ to $z$, comparison with $s\mapsto \frac st z$ yields
\begin{equation}
\label{eq:segment-}
\int^t_0\big|\dot{\tilde{\xi}}_+-\dot{\tilde{\xi}}_-\big|^2\ d\tau\geqslant \frac {|z|^2}t.
\end{equation}
Similarly,
\begin{equation*}
\int^t_0\big|\dot{\tilde{\xi}}_++\dot{\tilde{\xi}}_-\big|^2\ d\tau\geqslant \frac {|y-x|^2}t.
\end{equation*}
So,
\begin{equation}\label{eq:I_4}
I_4\geqslant \frac {\nu_\lambda h^2}{9t^3}(|z|^2+|y-x|^2).
\end{equation}
As for $I_5$, we have
$$
\begin{aligned}
	I_5=&
	\frac h{t-h}\int^t_0\big\langle L_v(\tilde{\xi}_-,\dot{\tilde{\xi}}_-),\dot{\tilde{\xi}}_--\dot{\tilde{\xi}}_+\big\rangle d\tau 
	+\Big(\frac h{t-h}-\frac h{t+h}\Big) \int^t_0\big\langle L_v(\tilde{\xi}_-,\dot{\tilde{\xi}}_-),\dot{\tilde{\xi}}_+\big\rangle d\tau\\
	&+\frac h{t+h}\int^t_0\big\langle L_v(\tilde{\xi}_-,\dot{\tilde{\xi}}_-)-L_v(\tilde{\xi}_-,\dot{\tilde{\xi}}_+),\dot{\tilde{\xi}}_+\big\rangle\ d\tau\\
	&+\frac h{t+h}\int^t_0\big\langle L_v(\tilde{\xi}_-,\dot{\tilde{\xi}}_+)-L_v(\tilde{\xi}_+,\dot{\tilde{\xi}}_+),\dot{\tilde{\xi}}_+\big\rangle\ d\tau.
\end{aligned}
$$
Thus, by \eqref{eq:UB} and Lemma \ref{tilde_three_acrs} we have
\begin{align*}
	I_5\geqslant& -C_2\frac{2h}t\int^t_0|\dot{\tilde{\xi}}_--\dot{\tilde{\xi}}_+|\ d\tau-C_2\kappa(4\lambda)\frac {2h^2t}{t^2-h^2}
	\\
	&-C_2\kappa(4\lambda)\frac{h}{t+h}\Big\{\int^t_0|\dot{\tilde{\xi}}_--\dot{\tilde{\xi}}_+|\ d\tau
	+\int^t_0|\tilde{\xi}_--\tilde{\xi}_+|\ d\tau\Big\}
	\\
	\geqslant&-4C_2\kappa(4\lambda)\frac{h^2}{t}-C_2\big[2+\kappa(4\lambda)\big]\frac{h}{\sqrt{t}} \Big(\int^t_0|\dot{\tilde{\xi}}_--\dot{\tilde{\xi}}_+|^2\ d\tau\Big)^{\frac 12} -C_2\kappa(4\lambda)h\sqrt{\frac{C_\lambda}t (h^2+|z|^2)} \\
	\geqslant&-C_3\frac{h}t (h+|z|)
\end{align*}
for some constant $C_3=C_3(\lambda)>0$. By the last inequality, \eqref{eq:I_3true}, and \eqref{eq:I_4} we get
\begin{equation}\label{eq:I_3bis}
I_3\geqslant \frac {\nu_\lambda h^2}{9t^3}(|z|^2+|y-x|^2)-C_3\frac{h}t (h+|z|).
\end{equation}

We are now ready to prove \eqref{eq:semiconvexity}. Indeed, by combining \eqref{eq:I_11}, \eqref{eq:I_2split}, and \eqref{eq:I_3bis}  we obtain, thanks to  Lemma ~\ref{tilde_three_acrs},
\begin{eqnarray*}
\lefteqn{A_{t+h}(x,y+z)+A_{t-h}(x,y-z)-2A_t(x,y) \geqslant -\frac{C_1}{t}(h^2+|z|^2)}
\\
& &-C_2\int^t_0 \big(\big|\tilde{\xi}_+-\tilde{\xi}_-\big|^2+\big|\dot{\tilde{\xi}}_+-\dot{\tilde{\xi}}_-\big|\cdot\big|\tilde{\xi}_+-\tilde{\xi}_-\big|\big)\ d\tau -C_3\frac{h}t (h+|z|)
\\
& \geqslant&-\frac{C_1}{t}(h^2+|z|^2)
-\frac{C_2}{2}\Big(3C_\lambda+\frac{C'_\lambda}{t}\Big)(h^2+|z|^2)-\frac{C_3}t \Big(\frac 32 h^2+|z|^2\Big).
\end{eqnarray*}

Finally, in order to prove \eqref{eq:convexity_local} observe that taking $h=0$ in \eqref{eq:I_1} and \eqref{eq:I_3bis} we conclude that $I_1,I_3\geqslant 0$. Therefore,  for any $\varepsilon>0$, \eqref{eq:I_2split} yields the lower bound 
\begin{eqnarray*}
\lefteqn{A_{t}(x,y+z)+A_{t}(x,y-z)-2A_t(x,y)}
\\
& \geqslant &\frac{\nu_\lambda}{4}\int^t_0\big|\dot{\tilde{\xi}}_+-\dot{\tilde{\xi}}_-\big|^2\ d\tau-C_2\int^t_0 \big(\big|\tilde{\xi}_+-\tilde{\xi}_-\big|^2+\big|\dot{\tilde{\xi}}_+-\dot{\tilde{\xi}}_-\big|\cdot\big|\tilde{\xi}_+-\tilde{\xi}_-\big|\big)\ d\tau
\\
& \geqslant &\Big(\frac{\nu_\lambda}{4}-\frac{\varepsilon}{2}\Big)\int^t_0\big|\dot{\tilde{\xi}}_+-\dot{\tilde{\xi}}_-\big|^2\ d\tau
-\Big(C_2+\frac{C_2^2}{2 \varepsilon}\Big)\int^t_0 \big|\tilde{\xi}_+-\tilde{\xi}_-\big|^2\ d\tau.
\end{eqnarray*}
Hence, taking $\varepsilon=\nu_\lambda/4$, recalling \eqref{eq:segment-},  and appealing to Lemma ~\ref{tilde_three_acrs} it follows that
\begin{equation*}
A_{t}(x,y+z)+A_{t}(x,y-z)-2A_t(x,y)\geqslant  \Big\{\frac{\nu_\lambda}{8t}-\Big(C_2+\frac{2C_2^2}{\nu_\lambda}\Big)C_\lambda\Big\}|z|^2.
\end{equation*}
Now, choosing $t''_{\lambda}\in(0,t_\lambda']$ such that 
$$
\frac{\nu_\lambda}{8t''_{\lambda}}>\Big(C_2+\frac{2C_2^2}{\nu_\lambda}\Big)C_\lambda
$$
one completes the proof.
\end{proof}

\subsection{$C^{1,1}_{loc}$ regularity of the fundamental solution}

\begin{Pro}\label{C11_A_t}
Suppose $L$ is a Tonelli Lagrangian and, for any $\lambda>0$,  let  $t_\lambda'>0$ be the number given by Lemma~\ref{tilde_three_acrs}. 

Then, for any $x\in\R^n$ the functions $(t,y)\mapsto A_t(x,y)$ and $(t,y)\mapsto A_t(y,x)$ are of class $C^{1,1}_{\text{loc}}$ on the cone $S_{\lambda}(x,t_\lambda')$ defined in \eqref{cone}. 
Moreover, for all $(t,y)\in S(x,t_\lambda')$
\begin{align}
D_yA_t(x,y)=&L_v(\xi(t),\dot{\xi}(t)),\label{eq:diff_A_t_y}\\
D_xA_t(x,y)=&-L_v(\xi(0),\dot{\xi}(0)),\label{eq:diff_A_t_x}\\
D_tA_t(x,y)=&-E_{t,x,y},\label{eq:diff_A_t_t}
\end{align}
where $\xi\in\Gamma^t_{x,y}$ is the unique minimizer for $A_t(x,y)$ and 
$$E_{t,x,y}:=H(\xi(s),p(s))\qquad\forall \,s\in[0,t]$$
is the energy of the Hamiltonian trajectory $(\xi,p)$ with
\begin{equation}
\label{dual_arc}
p(s)=L_v(\xi(s),\dot{\xi}(s)).
\end{equation}
\end{Pro}

\begin{proof}
$C^{1,1}$-regularity on $S(x,t_\lambda')$ is a corollary of propositions~\ref{semiconcave_A_t} and \ref{convexity_A_t}. \eqref{eq:diff_A_t_y} follows from Lemma \ref{superdiff_fund_sol} and \eqref{eq:diff_A_t_x} can be proved by a similar argument.

For the proof of \eqref{eq:diff_A_t_t}, first, we define for $h>0$ small enough
$$
\xi_+(\tau)=\xi\Big(\frac{t}{t+h}\tau\Big),\quad\tau\in[0,t+h].
$$
Then
\begin{align*}
A_{t+h}(x,y)-A_t(x,y)\leqslant&\int^{t+h}_0L(\xi_+(\tau),\dot{\xi}_+(\tau))\ d\tau-\int^t_0L(\xi(s),\dot{\xi}(s))\ ds\\
=&\frac{t+h}t\int^t_0L\Big(\xi(s),\frac{t}{t+h}\dot{\xi}(s)\Big)-\int^t_0L(\xi(s),\dot{\xi}(s))\ ds.
\end{align*}
Therefore
\begin{align*}
&\limsup_{h\to 0}\frac{A_{t+h}(x,y)-A_t(x,y)}{h}\\
\leqslant&\;\lim_{h\to 0}\frac 1t\int^t_0L\Big(\xi(s),\frac{t}{t+h}\dot{\xi}(s)\Big)ds
\\
&\quad+\lim_{h\to 0}\frac 1h\int^t_0\Big\{L\Big(\xi(s),\frac{t}{t+h}\dot{\xi}(s)\Big)-L(\xi(s),\dot{\xi}(s))\Big\}ds\\
=&\;\frac 1t\int^t_0\Big\{L(\xi(s),\dot{\xi}(s))-\big\langle L_v(\xi(s),\dot{\xi}(s)),\dot{\xi}(s)\big\rangle\Big\} ds=-\frac 1t\int^t_0H(\xi(s),p(s))\ ds,
\end{align*}
where $p(\cdot)$ is given by \eqref{dual_arc}.
The study of the function $(t,y)\mapsto A_t(y,x)$ is similar.
\end{proof}

\section{$\mathbf{y}$ is a generalized characteristic}\label{App.B}
The existence of singular generalized characteristics was first proved in \cite{Albano-Cannarsa} (see also \cite{Cannarsa-Sinestrari}) in a constructive way, and then a simplified proof was given in \cite{Yu} using an approximation method. Here, in order to prove that the curve $\mathbf{y}(t)$, $t\in[0,t_0]$, in Lemma \ref{maximizer_singular} is  a generalized characteristic, we follow the idea of the original proof from \cite{Albano-Cannarsa}. We sketch the proof  for completeness. The following result is an analogy to Lemma 5.5.6 in \cite{Cannarsa-Sinestrari}.

\begin{Lem}\label{B_1}
Let $L$ be a Tonelli Lagrangian, and $t_0\in(0,1]$ be given by Lemma \ref{maximizer_singular}. For any fixed $x\in\R^n$, let $\mathbf{y}:[0,t_0]\to\R^n$ be the curve constructed in Lemma \ref{maximizer_singular}, and let $\mathbf{p}:[0,t_0]\to\R^n$ be the arc defined in Proposition \ref{Main_lemma_g_c}. Then, for any $\varepsilon>0$, there exist arcs $\mathbf{x}_{\varepsilon}:[0,t_0]\to\R^n$, $\mathbf{p}_{\varepsilon}:[0,t_0]\to\R^n$, with $\mathbf{x}_{\varepsilon}(0)=x$ and $\mathbf{p}_{\varepsilon}(0)=p_x$ where $p_x$ is the unique element in $\arg\min_{p\in D^+u(x)}H(x,p)$, and a partition $0=s_0<s_1<\cdot<s_{k-1}<s_k=t_0$ with the following properties:
\begin{enumerate}[\rm (i)]
  \item $\max\{|s_{j+1}-s_j|: 0\leqslant j\leqslant k-1\}<\varepsilon$;
  \item $\mathbf{p}_{\varepsilon}(s)\in D^+u(\mathbf{x}_{\varepsilon}(s))$ for all $s\in[0,t_0]$;
  \item For each $j\in\{0,\ldots,k-1\}$, $\mathbf{p}_{\varepsilon}(\cdot)$ is continuous on $[s_j,s_{j+1})$;
  \item $|\mathbf{x}_{\varepsilon}(t)-\mathbf{x}_{\varepsilon}(s)|\leqslant C_1|t-s|$ for all $t,s\in[0,t_0]$;
  \item $|\mathbf{p}_{\varepsilon}(s)-\mathbf{p}_{\varepsilon}(s')|\leqslant C_2\sqrt{\varepsilon}$\, for all $s,s'\in[s_j,s_{j+1})$, $j\in\{0,\ldots,k-1\}$.
  \item The number of the nodes satisfies $k\leqslant [\frac {C_3}{\varepsilon}]+1$.
\end{enumerate}
\end{Lem}

\begin{proof}
	The construction of the arcs $\mathbf{x}_{\varepsilon}$ and $\mathbf{p}_{\varepsilon}$ follows the reasoning of Lemma 5.5.6 in \cite{Cannarsa-Sinestrari}. The proof of properties (i)-(vi) also goes as in \cite{Cannarsa-Sinestrari}, except for the use of the essential inequality (5.64) in \cite{Cannarsa-Sinestrari}, which can be replaced by property (d) in Proposition \ref{Main_lemma_g_c}.
\end{proof}

\begin{Lem}\label{Lem_B1}
	Let the arcs $\mathbf{x}_{\varepsilon}$ and $\mathbf{p}_{\varepsilon}$ be defined as in Lemma \ref{B_1}. Then there exists a constant $C>0$, independent of $\varepsilon$, such that 
	\begin{equation}\label{eq:B}
		\left|\mathbf{x}_{\varepsilon}(r)-\mathbf{x}_{\varepsilon}(s)-\int^r_sH_p(\mathbf{x}_{\varepsilon}(\tau),\mathbf{p}_{\varepsilon}(\tau))\ d\tau\right|\leqslant C\sqrt{\varepsilon}.
	\end{equation}
\end{Lem}

\begin{proof}
	In view of property (vi) in Lemma \ref{B_1}, it is sufficient to prove \eqref{eq:B} under the extra assumption that $s=0$ and $r<s_1$. Thus, we have
	\begin{align*}
		&\left|\mathbf{x}_{\varepsilon}(r)-\mathbf{x}_{\varepsilon}(s)-\int^r_sH_p(\mathbf{x}_{\varepsilon}(\tau),\mathbf{p}_{\varepsilon}(\tau))\ d\tau\right|=\left|\mathbf{y}(r)-\mathbf{y}(0)-\int^{r}_{0}H_p(\mathbf{y}(\tau),\mathbf{p}(\tau))\ d\tau\right|\\
		=&\left|\xi_{r}(r)-x-\int^r_0H_p(\xi_{\tau}(\tau),p_{\tau}(\tau))\ d\tau\right|=\left|\int^r_0\dot{\xi}_{r}(\tau)-H_p(\xi_{\tau}(\tau),p_{\tau}(\tau))\ d\tau\right|\\
		=&\left|\int^r_0H_p(\xi_{r}(\tau),p_{r}(\tau))-H_p(\xi_{\tau}(\tau),p_{\tau}(\tau))\ d\tau\right|,
	\end{align*}
	where $\xi_t\in\Gamma^t_{x,\mathbf{y}(t)}$ is a minimizer for $A_t(x,\mathbf{y}(t))$, $t\in(0,s_1)$ and $p_t(\tau):=L_v(\xi_t(\tau),\dot{\xi}_t(\tau))$, $\tau\in[0,t]$. Then, by Lemma \ref{uni_Lip_1}, Proposition \ref{Main_lemma_g_c} (a) and Lemma \ref{B_1}, we have
	\begin{align*}
		&|H_p(\xi_{r}(\tau),p_{r}(\tau))-H_p(\xi_{\tau}(\tau),p_{\tau}(\tau))|\\
		\leqslant&|H_p(\xi_{r}(\tau),p_{r}(\tau))-H_p(\xi_{r}(\tau),p_{\tau}(\tau))|+|H_p(\xi_{r}(\tau),p_{\tau}(\tau))-H_p(\xi_{\tau}(\tau),p_{\tau}(\tau))|\\
		\leqslant& C_1|p_{r}(\tau))-p_{\tau}(\tau))|+C_2|\xi_{r}(\tau)-\xi_{\tau}(\tau)|\\
		\leqslant&C_1(|p_{r}(\tau))-p_{r}(r))|+|p_{r}(r))-p_{\tau}(\tau))|)+C_2(|\xi_{r}(\tau)-\xi_{r}(r)|+|\xi_{r}(r)-\xi_{\tau}(\tau)|)\\
		\leqslant&C_1(C_3\varepsilon+|\mathbf{p}(r)-\mathbf{p}(\tau)|)+C_2(C_4\varepsilon+|\mathbf{y}(r')-\mathbf{y}(\tau)|)\\
		\leqslant&C_5\varepsilon+C_6\sqrt{\varepsilon}\leqslant C_7\sqrt{\varepsilon},
	\end{align*}
	which leads to \eqref{eq:B}.
\end{proof}

The rest  of the proof is standard, see, e.g., \cite{Albano-Cannarsa} or \cite{Cannarsa-Sinestrari}. As $\varepsilon\to0$ in \eqref{eq:B}, we obtain
$$
\dot{\mathbf{y}}(s)\in\text{co}\,H_p(\mathbf{y}(s),D^+u(\mathbf{y}(s))).
$$
We omit the rest of the proof. The reader can refer to, for instance, \cite[Page 133-135]{Cannarsa-Sinestrari}.

\section{Global viscosity solutions on $\R^n$}
\label{appendix:B}
In this section we prove Proposition~\ref{Ext_and_reachable}.

\begin{proof}
The first part of the conclusion, that is, the fact that there exists a constant $c(H)\in\R$ such that the Hamilton-Jacobi equation
\eqref{eq:crtical_H_J_eqn}
admits a viscosity solution $u:\R^n\to\R$  for $c=c(H)$ and does not admit any such solution for $c<c(H)$ is guaranteed by Theorem~1.1 in \cite{Fathi-Maderna}.
	Moreover, in view of Proposition~4.1 in \cite{Fathi-Maderna}, we have that 
	$$
	u=T^-_tu+c(H)t\qquad\forall\ t\geqslant 0,
	$$ 
where $T^-_t$ is defined in \eqref{L-L regularity_inf}. Therefore, $u$ is Lipschitz continuous on $\R^n$ on account of Proposition~3.2 in \cite{Fathi-Maderna}. 

We proceed to show that $u$ is also semiconcave. Let $A_t(x,y)$ be the fundamental solution of \eqref{eq:crtical_H_J_eqn} with $c=c(H)$ and fix $t_0\in (0,2/3)$. For every $x\in\R^n$ we have that
\begin{eqnarray*}
u(x)=\min_{y\in\R^n}\big\{u(y)+A_{t_0}(y,x)\big\} +c(H)t_0.
\end{eqnarray*}
Let $y_x\in\R^n$ be a point at which the above minimum is attained. Then taking $\lambda=1$ in Proposition \ref{semiconcave_A_t} we conclude that, for all $y\in B(y_x, t_0)$ and all $z\in B(0, t_0)$,
\begin{equation*}
A_{t_0}(y_x,y+z)+A_{t_0}(y_x,y-z)-2A_{t_0}(y_x,y)\leqslant \frac{C_1}{t_0}|z|^2
\end{equation*}
for some constant $C_1>0$ independent of $y_x$.
 	Therefore, taking $y=x$ in the above inequality we obtain
	\begin{eqnarray*}
\lefteqn{u(x+z)+u(x-z)-2u(x)}
\\
&\leqslant& A_{t_0}(y_x,x+z)+A_{t_0}(y_x,x-z)-2A_{t_0}(y_x,x)\leqslant \frac{C_1}{t_0}|z|^2
\end{eqnarray*}
all $z\in B(0, t_0)$. So, being Lipschitz, $u$ is semiconcave on $\R^n$. \end{proof}

\end{document}